\documentclass{amsart}

\usepackage{amsmath,amssymb,verbatim,lpic}
\usepackage{amsthm}
\usepackage{stmaryrd}
\usepackage{enumerate,multicol}
\usepackage{url}

\newtheorem{theorem}{Theorem}[section]
\newtheorem{proposition}[theorem]{Proposition}
\newtheorem{lemma}[theorem]{Lemma}
\newtheorem{corollary}[theorem]{Corollary}

\theoremstyle{definition}
\newtheorem{definition}[theorem]{Definition}
\newtheorem{fact}[theorem]{Fact}

\newtheorem{question}[theorem]{Question}

\newtheorem{example}[theorem]{Example}
\newtheorem{exercise}[theorem]{Exercise}

\newcommand{\abar}{\bar{a}}
\newcommand{\bbar}{\bar{b}}
\newcommand{\cbar}{\bar{c}}
\newcommand{\dbar}{\bar{d}}
\newcommand{\ibar}{\bar{\imath}}
\newcommand{\jbar}{\bar{\jmath}}
\newcommand{\kbar}{\bar{k}}
\newcommand{\ubar}{\bar{u}}
\newcommand{\vbar}{\bar{v}}
\newcommand{\ybar}{\bar{y}}
\newcommand{\xbar}{\bar{x}}

\newcommand{\cI}{\mathcal{I}}
\newcommand{\cL}{\mathcal{L}}
\newcommand{\cU}{\mathcal{U}}

\newcommand{\seq}{\subseteq}
\newcommand{\vphi}{\varphi}

\def\M{\mathbb{M}}

\def\U{\mathbb{U}}

\def\acl{\operatorname{acl}}
\def\dcl{\operatorname{dcl}}
\def\dom{\operatorname{dom}}
\newcommand{\claim}{$_{\sslash}$}
\def\NSOP{\operatorname{NSOP}}
\def\NTP{\operatorname{NTP}}
\def\SOP{\operatorname{SOP}}
\def\Th{\operatorname{Th}}
\def\tp{\operatorname{tp}}
\def\TP{\operatorname{TP}}
\def\IS{\operatorname{Ind}}
\def\EM{\operatorname{EM}}

\def\Ind{\setbox0=\hbox{$x$}\kern\wd0\hbox to 0pt{\hss$\mid$\hss}
\lower.9\ht0\hbox to 0pt{\hss$\smile$\hss}\kern\wd0}

\def\Notind{\setbox0=\hbox{$x$}\kern\wd0\hbox to 0pt{\mathchardef
\nn=12854\hss$\nn$\kern1.4\wd0\hss}\hbox to
0pt{\hss$\mid$\hss}\lower.9\ht0 \hbox to 0pt{\hss$\smile$\hss}\kern\wd0}

\def\ind{\mathop{\mathpalette\Ind{}}}

\def\nind{\mathop{\mathpalette\Notind{}}}

\newcommand{\dotminus}{ 
\!\!\buildrel\textstyle~.\over{\hbox{ 
\vrule height3pt depth0pt width0pt}{\smash-} 
}}

\begin{document}

\title{Model theoretic properties of the Urysohn sphere}

\author{Gabriel Conant}

\address{Department of Mathematics, Statistics and Computer Science\\
University of Illinois at Chicago\\
Chicago, IL, 60607, USA}

\email{gconan2@uic.edu}


\author{Caroline Terry}

\address{Department of Mathematics, Statistics and Computer Science\\
University of Illinois at Chicago\\
Chicago, IL, 60607, USA}

\email{cterry3@uic.edu}

\begin{abstract}
We characterize model theoretic properties of the Urysohn sphere as a metric structure in continuous logic. In particular, our first main result shows that the theory of the Urysohn sphere is $\SOP_n$ for all $n\geq 3$, but does not have the fully finite strong order property. Our second main result is a geometric characterization of dividing independence in the theory of the Urysohn sphere. We further show that this characterization satisfies the extension axiom, and so forking and dividing are the same for complete types. Our results require continuous analogs of several tools and notions in classification theory. While many of these results are undoubtedly known to researchers in the field, they have not previously appeared in publication. Therefore, we include a full exposition of these results for general continuous theories. 
\end{abstract}

\maketitle

\section{Introduction}

The Urysohn sphere is the unique complete separable metric space of diameter 1, which is ultrahomogeneous and isometrically embeds every separable metric space of diameter $\leq 1$. As a metric space, the Urysohn sphere is an important example in descriptive set theory, infinitary Ramsey theory, and topological dynamics of automorphism groups. As a model theoretic structure, the Urysohn sphere is most naturally studied by way of continuous logic. Indeed, the theory of the Urysohn sphere can be considered as the model completion of the ``empty theory" in the continuous language containing only a predicate for the metric. As such, the Urysohn sphere is often used as a fundamental example of the kind of structure well-suited for study in continuous logic. Previous work on the model theory of the Urysohn sphere can be found in \cite{EaGo}, \cite{GoUS}, and \cite{Usvy}.

In \cite{EaGo}, Goldbring and Ealy characterize thorn-forking in the Urysohn sphere and show that the theory is rosy (with respect to finitary imaginaries). They also include an argument, due to Pillay, that the Urysohn sphere is not simple, and mention unpublished observations of Berenstein and Usvyatsov that the Urysohn sphere is $\SOP_3$, but without the strict order property. Altogether, this previous work motivates the following questions, which we answer in this paper.
\begin{enumerate}
\item How complicated is the theory of the Urysohn sphere with respect to commonly considered model theoretic dividing lines?
\item What is the nature of forking independence in this theory?
\end{enumerate}
In particular, our first main result shows that the Urysohn sphere has the $n$-strong order property $(\SOP_n)$ for all $n\geq 3$, but does not have the fully finite strong order property. This significantly strengthens the previously known classification. Moreover, answering a question of Starchenko, we show that the Urysohn sphere has the tree property of the second kind. In our second main result, we characterize forking and dividing for complete types in terms of basic distance calculations on metric spaces. A corollary of this characterization is that forking and dividing are the same for complete types, which indicates some good behavior of nonforking despite the model theoretic complexity of the theory. These results answer questions posed by Goldbring and Starchenko. 

The previous results require an understanding of several notions concerning classification theory and forking and dividing. Therefore, this paper also serves to organize and verify continuous versions of several standard facts in this area. 

First, we formulate continuous definitions of Shelah's strong order properties, and prove equivalent versions using amalgamation of indiscernible sequences. We then formulate definitions of forking and dividing for continuous logic, which are directly adapted from the classical ``syntactic" definitions. We also state equivalent versions of these definitions, which are familiar from sources in both classical and continuous model theory. In \ref{appMT}, we adapt the classical proofs of these equivalences to continuous logic. In particular, we define dividing via $k$-inconsistency of infinite sequences and prove the equivalence of this definition with the more common version, found in \cite{BBHU}, which uses inconsistency of indiscernible sequences. We also give a syntactic definition of forking resulting from a continuous version of ``implying a disjunction of dividing formulas", and prove its equivalence with the definition motivated by the existence of nonforking extensions.

A common tool in the previous results is the ability to assume indiscernibility for infinite sequences, which witness certain model theoretic behavior. In classical logic, this is done by taking indiscernible realizations of the \textit{Ehrenfeucht-Mostowski type}. The existence of such realizations is a standard application of Ramsey's Theorem. In \ref{appMT}, we define the $\EM$-type for continuous logic, and prove that Ramsey's Theorem can be used to obtain indiscernible sequences realizing these types.  

The outline of the paper is as follows. In Section \ref{sec:MT} we translate Shelah's $\SOP_n$-hierarchy to continuous logic, and prove that these dividing lines are detected by \textit{$n$-cyclic} indiscernible sequences, i.e. indiscernible sequences whose $2$-type can be consistently amalgamated in an $n$-cycle (see Definition \ref{def:SOP}). We then define forking and dividing and prove, for continuous logic, a standard fact that the equivalence of these notions is witnessed by the extension axiom for nondividing. Section \ref{sec:main} contains our main results concerning the the theory of the Urysohn sphere, denoted $\cU$. We first define $\cU$, and recall the result, due to Henson, that $\Th(\cU)$ is separably categorical and eliminates quantifiers. We then give a classification of the model theoretic complexity of the Urysohn sphere. In particular, we show that $\Th(\cU)$ is $\SOP_n$ for all $n\geq 3$, but does not have the fully finite strong order property. We also prove that $\Th(\cU)$ has $\TP_2$. We then turn our attention to nonforking independence. We first give a combinatorial characterization of dividing, which is formulated from basic distance calculations in metric spaces. We then show that this characterization satisfies the extension axiom, and so forking and dividing are the same for complete types. Using the characterization of nonforking, we show that a type in $\Th(\cU)$ is stationary if and only if it is algebraic. In Section \ref{sec:Qs}, we discuss further remarks and questions. Specifically, we observe a relationship between our characterization of nonforking in the Urysohn sphere and the stationary independence relation given by free amalgamation of metric spaces, which was used by Tent and Ziegler to prove the that the isometry group of the Urysohn space is (algebraically) simple. This motivates several open questions (in both discrete and continuous logic) on the relationship between stationary independence relations and the equivalence of forking and dividing for complete types. 
 
\subsection*{Acknowledgments} 

We would like to thank Isaac Goldbring for introducing this project to us, and for guidance throughout the development of our results. We also thank Dave Marker for many helpful conversations.

\section{Continuous Model Theory}\label{sec:MT}

We assume the reader is familiar with the basic setting of continuous logic and model theory for bounded metric structures. An in-depth introduction can be found in \cite{BBHU}. Throughout this section, $T$ denotes a complete theory in continuous logic, and $\M$ is a sufficiently saturated monster model of $T$. The letters $A,B,C,\ldots$ denote sets, and we write $A\subset\M$ to mean $A\seq\M$ and $\M$ is $\chi(A)^+$-saturated, where $\chi(A)$ is the density character of $A$. We will use $\abar,\bbar,\cbar,\ldots$ to denote tuples of elements, and $a,b,c,\ldots$ to denote singletons. We use $\ell(\abar)$ to denote the length of a tuple, which may be infinite.

Recall that types in continuous logic consist of conditions of the form ``$\vphi(\xbar)=0$", where $\vphi(\xbar)$ is some formula. Given $\epsilon>0$, we use $``\vphi(\xbar)\leq\epsilon$" and $``\vphi(\xbar)\geq\epsilon"$ to denote, respectively, the conditions ``$\vphi(\xbar)\dotminus\epsilon=0$" and $``\epsilon\dotminus\vphi(\xbar)=0"$, where, given $r,s\in[0,1]$, $r\dotminus s=\max\{0,r-s\}$. We use ``$\vphi(\xbar)=\epsilon$" to denote the condition ``$|\vphi(\xbar)-\epsilon|=0$".

We will exclusively study metric structures of diameter $\leq1$. Therefore, throughout this paper, when carrying out calculations with distances we use addition truncated at 1. We also adopt the convention that $\sup\emptyset=0$ and $\inf\emptyset=1$. 

\subsection{Classification Theory}

In this section, we define the continuous analogs of several ``dividing lines" in Shelah's classification hierarchy. The translation of these properties to continuous logic is an ongoing process. For example, stability is discussed in \cite{BBHU} and definitions of the independence property, the tree property of the second kind, and the strict order property can be found in \cite{BYrandom}. Our focus will be on the strong order property, which was first defined for classical logic in \cite{Sh500}. Following the style of \cite{BYrandom} and \cite{BBHU}, we give syntactic definitions of various strong order properties, which are obtained from the definitions in \cite{Sh500} via a standard transfer of discrete connectives to continuous ones (e.g. $\max$ in place of conjunction).

\begin{definition} $~$ 
\begin{enumerate}
\item Given $n\geq 3$, $T$ has the \textbf{$n$-strong order property}, $\SOP_n$, if there is a formula $\vphi(\xbar,\ybar)$, an $\epsilon>0$, and a sequence $(\abar^l)_{l<\omega}$ in $\M$, such that $\vphi(\abar^l,\abar^m)=0$ for all $l<m$ and
$$
\inf_{\xbar^1,\ldots,\xbar^n}\max\{\vphi(\xbar^1,\xbar^2),\ldots,\vphi(\xbar^{n-1},\xbar^n),\vphi(\xbar^n,\xbar^1)\}\geq\epsilon.
$$
\item $T$ has the \textbf{fully finite strong order property}, $\SOP_\infty$, if there is a formula $\vphi(\xbar,\ybar)$, an $\epsilon>0$, and a sequence $(\abar^l)_{l<\omega}$ in $\M$, such that $\vphi(\abar^l,\abar^m)=0$ for all $l<m$ and for all $n$
$$
\inf_{\xbar^1,\ldots,\xbar^n}\max\{\vphi(\xbar^1,\xbar^2),\ldots,\vphi(\xbar^{n-1},\xbar^n),\vphi(\xbar^n,\xbar^1)\}\geq\epsilon.
$$
\end{enumerate}
\end{definition}

When working with such notions, it is often convenient to obtain some level of indiscernibility in the definitions. For $\SOP_n$, we have the following result.

\begin{definition}\label{def:SOP}
Suppose $(\abar^l)_{l<\omega}$ is an indiscernible sequence and $p(\xbar,\ybar)=\tp(\abar^0,\abar^1)$. Given $n\geq 1$, $(\abar^l)_{l<\omega}$ is \textbf{$n$-cyclic} if
$$
p(\xbar^1,\xbar^2)\cup p(\xbar^2,\xbar^3)\cup\ldots\cup p(\xbar^{n-1},\xbar^n)\cup p(\xbar^n,\xbar^1)
$$
is satisfiable.
\end{definition}

\begin{proposition}\label{SOP}
$~$
\begin{enumerate}[$(a)$]
\item Given $n\geq 3$, $T$ has $\SOP_n$ if and only if there is an indiscernible sequence in $\M$ that is not $n$-cyclic.
\item If $T$ has $\SOP_\infty$ then there is an indiscernible sequence of finite tuples in $\M$ that is not $n$-cyclic for any $n\geq 1$.
\end{enumerate}
\end{proposition}
\begin{proof}
We prove part $(a)$ and leave the proof of part $(b)$, which is similar, to the reader. Suppose first that $T$ has $\SOP_n$, witnessed by a formula $\vphi(\xbar,\ybar)$, an $\epsilon>0$, and a sequence $(\abar^l)_{l<\omega}$. Let $(\bbar^l)_{l<\omega}$ be an indiscernible sequence realizing the EM-type of $(\abar^l)_{l<\omega}$ (see Section \ref{EM} of \ref{appMT}). We show that $(\bbar^l)_{l<\omega}$ is not $n$-cyclic. Since $\vphi(\abar^l,\abar^m)=0$ for all $l<m$, it follows that the EM-type of $(\abar^l)_{l<\omega}$ contains the  condition ``$\vphi(\xbar^0,\xbar^1)=0$", and so $\vphi(\bbar^0,\bbar^1)=0$. Let $p(\xbar,\ybar)=\tp(\bbar^0,\bbar^1)$ and suppose, toward a contradiction, that $(\cbar^1,\ldots,\cbar^n)\models p(\xbar^1,\xbar^2)\cup\ldots\cup p(\xbar^{n-1},\xbar^n)\cup p(\xbar^n,\xbar^1)$. Then we have $\max\{\vphi(\cbar^1,\cbar^2),\ldots,\vphi(\cbar^{n-1},\cbar^n),\vphi(\cbar^n,\cbar^1)\}=0$, which contradicts
$$
\inf_{\xbar^1,\ldots,\xbar^n}\max\{\vphi(\xbar^1,\xbar^2),\ldots,\vphi(\xbar^{n-1},\xbar^n),\vphi(\xbar^n,\xbar^1)\}\geq\epsilon.
$$

Conversely, suppose there is an indiscernible sequence $(\abar^l)_{l<\omega}$ in $\M$, which is not $n$-cyclic. Let $p(\xbar,\ybar)=\tp(\abar^0,\abar^1)$. By compactness, there is a formula $\vphi(\xbar,\ybar)$ such that the condition ``$\vphi(\xbar,\ybar)=0$" is in $p(\xbar,\ybar)$, and $\{\vphi(\xbar^1,\xbar^2)=0,\ldots,\vphi(\xbar^{n-1},\xbar^n)=0,\vphi(\xbar^n,\xbar^1)=0\}$ is unsatisfiable. By compactness, we may fix $\epsilon>0$ such that
$$
\inf_{\xbar^1,\ldots,\xbar^n}\max\{\vphi(\xbar^1,\xbar^2),\ldots,\vphi(\xbar^{n-1},\xbar^n),\vphi(\xbar^n,\xbar^1)\}\geq\epsilon,
$$
and so $\vphi(\xbar,\ybar)$, together with $\epsilon$ and $(\abar^l)_{l<\omega}$, witness that $T$ has $\SOP_n$.
\end{proof}

Note that both Definition \ref{def:SOP}, as well as the statement of Proposition \ref{SOP}, can be stated for discrete logic without any alteration. In the discrete case, the previous result was first pointed out to us by Lynn Scow.

From the definitions, it is clear that if $T$ has $\SOP_\infty$ then it has $\SOP_n$ for all $n\geq 3$. Moreover, if $T$ has $\SOP_{n+1}$ then it has $\SOP_n$. Indeed, if $(\abar^l)_{l<\omega}$ is an indiscernible sequence, which is not $(n+1)$-cyclic, then one can show that $(\abar^{2l},\abar^{2l+1})_{l<\omega}$ is an indiscernible sequence, which is not $n$-cyclic. 

Finally, we recall the definition of the tree property of second kind. Our version follows the definition for classical logic from \cite{ChKa}. An equivalent formulation, which uses mutually indiscernible arrays, is given for continuous logic in \cite{BYrandom}. 

\begin{definition}
A theory has the \textbf{tree property of the second kind}, $\TP_2$, if there is a formula $\vphi(\xbar,\ybar)$, an integer $k>0$, and an array $(\abar_{i,j})_{i,j<\omega}$ such that
\begin{enumerate}[$(i)$]
\item for all $\sigma\in\omega^\omega$, $\{\vphi(\xbar,\abar_{n,\sigma(n)})=0:n<\omega\}$ is satisfiable,
\item for all $n<\omega$ and $\{\vphi(\xbar,\abar_{n,i})=0: i<\omega\}$ is $k$-unsatisfiable.
\end{enumerate}
\end{definition}

\subsection{Forking and Dividing}

We now turn to forking and dividing in  continuous logic. We emphasize that these notions have been previously formulated and frequently studied in the continuous setting. For example, the definition of dividing in \cite{BBHU} is identical to the statement of Theorem \ref{FANDD}$(a)$ below. On the other hand, there does not seem to be a definitive source for a general definition of forking in continuous logic, since most work has focused on stable and dependent theories. Therefore, we have chosen to start with definitions that most closely resemble their standard syntactic counterparts in classical logic. We will then discuss the equivalence of these definitions with those appearing previous literature.

Given a (possibly incomplete) type $\pi(\xbar)$ and a formula $\vphi(\xbar)$, we write $\pi(\xbar)\models``\vphi(\xbar)=0"$ if $\vphi(\abar)=0$ for every realization $\abar$ of $\pi(\xbar)$.

\begin{definition}
Fix $C\subset\M$.
\begin{enumerate}
\item Fix a formula $\vphi(\xbar,\ybar)$ and a tuple $\bbar\in\M$. 
\begin{enumerate}[$(a)$]
\item $\vphi(\xbar,\bbar)$ \textbf{divides over $C$} if there is a sequence $(\bbar^l)_{l<\omega}$, an $\epsilon>0$, and an integer $k>0$ such that $\bbar^l\equiv_C\bbar$ for all $l<\omega$ and $\{\vphi(\xbar,\bbar^l)\leq\epsilon:l<\omega\}$ is $k$-unsatisfiable.
\item $\vphi(\xbar,\bbar)$ \textbf{forks over $C$} if there are formulas $\psi_1(\xbar,\bbar_1),\ldots,\psi_m(\xbar,\bbar_m)$ such that each $\psi_i(\xbar,\bbar_i)$ divides over $C$ and 
$$
\{\vphi(\xbar,\bbar)=0\}\models ``\min\{\psi_1(\xbar,\bbar_1),\ldots,\psi_m(\xbar,\bbar_m)\}=0".
$$
\end{enumerate}
\item A type $\pi(\xbar)$ \textbf{forks} (resp. \textbf{divides}) \textbf{over $C$} if there is a formula $\vphi(\xbar,\bbar)$, which forks (resp. divides) over $C$, such that $\pi(\xbar)\models``\vphi(\xbar,\bbar)=0"$.
\end{enumerate}
\end{definition}

We start with the following useful lemma.

\begin{lemma}\label{divprop2}
If $\vphi(x,\bbar)$ divides over $C\subset\M$, then there is some $\epsilon>0$ such that $\vphi(x,\bbar)\dotminus\epsilon$ divides over $C$.
\end{lemma}
\begin{proof}
Fix a sequence $(\bbar^l)_{l<\omega}$, an integer $k>0$, and some $\epsilon>0$, which witness that $\vphi(x,\bbar)$ divides over $C$. Then $(\bbar^l)_{l<\omega}$, $k$, and $\frac{\epsilon}{2}$ witness that $\vphi(x,b)\dotminus\frac{\epsilon}{2}$ divides over $C$.
\end{proof}

Next, we state equivalent formulations of these notions, which are familiar from classical logic and will be much easier to work with. The proof of these equivalences is similar to the discrete case, and we outline the argument in \ref{appMT}.

\begin{theorem}\label{FANDD}
Fix $C\subset\M$ and a partial type $\pi(\xbar,\bbar)$. 
\begin{enumerate}[$(a)$]
\item $\pi(\xbar,\bbar)$ divides over $C$ if and only if there is a $C$-indiscernible sequence $(\bbar^l)_{l<\omega}$, with $\bbar^0=\bbar$, such that $\bigcup_{l<\omega}\pi(\xbar,\bbar^l)$ is unsatisfiable.
\item $\pi(\xbar,\bbar)$ forks over $C$ if and only if there is some $D\supseteq\bbar C$ such that any extension of $\pi(\xbar,\bbar)$, to a complete type over $D$, divides over $C$.
\end{enumerate}
\end{theorem}

\begin{definition} 
Define the following ternary relations on subsets of $\M$,
\begin{align*}
\textstyle A\ind^d_C B &\textnormal{ if and only if } \tp(A/BC)\textnormal{ does not divide over $C$},\\
\textstyle A\ind^f_C B &\textnormal{ if and only if }\tp(A/BC)\textnormal{ does not fork over $C$}.
\end{align*}
\end{definition}

We note the following basic fact that forking and dividing satisfy finite character. The proof follows from compactness exactly as in discrete logic.

\begin{fact}\label{basic}
Suppose $A,B,C\subset\M$. Then $\textstyle A\ind^d_C B$ if and only if $\textstyle\abar\ind^d_C \bbar$ for all finite $\abar\in A$ and $\bbar\in B$. The same is true with $\ind^d$ replaced by $\ind^f$.
\end{fact}

In \ref{appMT}, we verify that, as in classical logic, $\ind^f$ always satisfies extension. As a result, we have the following test for when forking and dividing are the same for complete types. 

\begin{theorem}\label{ext reduct} The following are equivalent.
\begin{enumerate}[$(i)$]
\item For all $A,B,C,\subset\M$, $A\ind^d_C B$ if and only if $A\ind^f_C B$.
\item The relation $\ind^d$ satisfies \textbf{extension}, i.e. for all $A,B,C,D\subset\M$, if $A\ind^d_C B$ and $BC\seq D$ then there is $A'\equiv_{BC} A$ such that $A'\ind^d_C D$.
\item\label{the reduction} For all $A,B,C\subset\M$, if $A\ind^d_C B$ and $b_*\in\M$ is a singleton then there is $A'\equiv_{BC}A$ such that $A'\ind^d_C Bb_*$. 
\end{enumerate}
\end{theorem}
\begin{proof}
The equivalence of $(i)$ and $(ii)$ can be shown exactly as in the discrete case, and we outline the proof in \ref{appMT}. The implication $(ii)\Rightarrow (iii)$ is trivial. Condition $(iii)$ is specifically formulated for later results, so we prove $(iii)\Rightarrow (ii)$. 

Assume $(iii)$ and suppose we have $A,B,C\subset\M$ such that $A\ind^d_C B$. Fix an enumeration $A=(a_i)_{i<\lambda}$ and let $\xbar=(x_i)_{i<\lambda}$ be a tuple of variables. Suppose $D\supseteq BC$. Define $\Sigma=\{\vphi(\xbar,\bbar):\vphi(\xbar,\ybar)\in\mathcal{L}_C,~\bbar\in D,~\vphi(\xbar,\bbar)\textnormal{ divides over $C$}\}$.

Given $\vphi\in\Sigma$ we use Lemma \ref{divprop2} to fix $\epsilon_\vphi>0$ such that $\vphi(\xbar,\bbar)\dotminus\epsilon_\vphi$ divides over $C$. Define $p(\xbar)=\tp_{\xbar}(A/BC)\cup\{\vphi(\xbar,\bbar)\geq\epsilon_\vphi:\vphi(\xbar,\bbar)\in\Sigma\}$. \\

\noindent\textit{Claim}: $p(\xbar)$ is satisfiable.

\noindent\textit{Proof}: By compactness, we may reduce to a finite subset $p_0(\xbar)\seq p(\xbar)$. Then $p_0(\xbar)$ is implied by a type of the form
$$
\pi(\xbar)=\tp_{\xbar}(A/BC)\cup\{\vphi(\xbar,\bbar,\dbar)\geq\epsilon_\vphi:\bbar\in B,~\vphi(\xbar,\bbar,\dbar)\textnormal{ divides over $C$}\},
$$
where $\dbar\in D$ is a fixed finite tuple.

Let $\dbar=(d_1,\ldots,d_n)$. Given $0\leq k\leq n$, suppose we have $A_k\equiv_{BC}A$ such that $A_k\ind^d_C B(d_i)_{i\leq k}$ (where the case $k=0$ is satisfied by $A_0=A$). By $(ii)$, there is $A_{k+1}\equiv_{BC(d_i)_{i\leq k}}A_k$ such that $A_{k+1}\ind^d_C B(d_i)_{i\leq k+1}$. Altogether, this constructs $A_n\equiv_{BC}A$ such that $A_n\ind^d_C B\dbar$. 

To finish the proof of the claim, we show $A_n\models\pi(\xbar)$. We have $A_n\equiv_{BC}A$ so if $A_n\not\models p(\xbar)$ then it follows that there is some $\vphi(\xbar,\bbar,\dbar)$, which divides over $C$, such that $\bbar\in B$ and $\vphi(A_n,\bbar,\dbar)<\epsilon_\vphi$. By choice of $\epsilon_\vphi$, this means $\tp(A_n/BC\dbar)$ divides over $C$, which contradicts the choice of $A_n$.\claim\\

By the claim, we may fix a realization $A'$ of $p(\xbar)$. We clearly have $A'\equiv_{BC} A$, and we want to show $A'\ind^d_C D$. If $A'\nind^d_C D$ then there is some formula $\vphi(\xbar,\bbar)\in\Sigma$, such that $``\vphi(\xbar,\bbar)=0"\in\tp(A'/D)$. But this contradicts $``\vphi(\xbar,\bbar)\geq\epsilon_\vphi"\in p(\xbar)$.
\end{proof}

\section{The Urysohn sphere}\label{sec:main}

We now turn to the main goals of this paper, which are characterizations of model theoretic properties of the Urysohn sphere. Before defining the Urysohn sphere, it is worth reiterating our conventions regarding metric spaces. In particular, we will only consider metric spaces with diameter bounded by $1$, and so, when carrying out calculations with distances, we truncate all operations at $1$. By convention, we set $\inf\emptyset=1$ and $\sup\emptyset=0$.

Following \cite{MeUS}, we recall that a separable metric space $X$ is \textbf{universal} if every separable metric space is isometric to a subspace of $X$, and \textbf{ultrahomogeneous} if every isometry between finite subspaces of $X$ extends to an isometry of $X$. 

We now define the Urysohn sphere, first constructed by Urysohn in \cite{Ury}.

\begin{definition}
The \textbf{Urysohn sphere}, $\cU$, is the unique complete separable metric space of diameter 1, which is ultrahomogeneous and universal for separable metric spaces of diameter $\leq1$.
\end{definition}

We will consider the Urysohn sphere as a metric structure in continuous logic. See \cite{EaGo} and \cite{Usvy} for important results about the Urysohn sphere as a continuous structure. For us, the salient points are the following.

\begin{enumerate}
\item We consider $\cU$ in the ``empty language" containing only the metric $d$.
\item The theory of $\cU$ is separably categorical and has quantifier elimination in this language (see \cite{Usvy}). Therefore complete types are entirely determined by distances. In particular, if $M\models\Th(\cU)$, $C\seq M$ and $\abar=(a_1,\ldots,a_n)\in M$ then $\tp(\abar/C)$ is completely determined by
$$
\{d(x_i,x_j)=d(a_i,a_j):1\leq i,j\leq n\}\cup\{d(x_i,c):1\leq i\leq n,~c\in C\}.
$$
\end{enumerate}

Throughout this section, $\U$ denotes a sufficiently saturated ``monster" model of $\Th(\cU)$. By saturation and quantifier elimination, we have the following fact.

\begin{proposition}\label{embed}
If $A\subset\U$ and $B$ is a metric space such that $A\seq B$ and $\U$ is $\chi(B)$-saturated, then $B$ isometrically embeds into $\U$ over $A$.
\end{proposition}

\subsection{Classification of $\Th(\cU)$}\label{SOPUS}

The goal of this section is to place $\Th(\mathcal{U})$ in the $\SOP_n$-hierarchy. We will use the characterization of $\SOP_n$ using cyclic indiscernible sequences (see Proposition \ref{SOP}). In particular, given an integer $n>0$, we will be interested in the satisfiability of types of the form
$$
p(\xbar^1,\xbar^2)\cup p(\xbar^2,\xbar^3)\cup\ldots\cup p(\xbar^{n-1},\xbar^n)\cup p(\xbar^n,\xbar^1),
$$
where $p(\xbar,\ybar)=\tp(\abar^0,\abar^1)$ for some indiscernible sequence $(\abar^l)_{l<\omega}$. By quantifier elimination, a type of this form is simply a partially defined metric space on the set of points $\xbar^1\cup\ldots\cup\xbar^n$. Therefore, we set forth some basic notions and facts regarding the completion of partially defined metric spaces to fully defined ones.

\begin{definition}
Suppose $X$ is set and $f:\dom(f)\seq X\times X\longrightarrow[0,1]$ is a symmetric partial function.
\begin{enumerate}
\item $f$ is a \textbf{partial semimetric} if, for all $x\in X$, $(x,x)\in\dom(f)$ and $f(x,x)=0$. In this case, $(X,f)$ is a \textbf{partial semimetric space}. We say $(X,f)$ is \textbf{consistent} if there is a pseudometric on $X$ extending $f$.
\item Given $m\geq 1$, a sequence $(x_0,x_1,\ldots,x_m)$ in $X^{m+1}$ is an \textbf{$f$-sequence} if $(x_0,x_m)\in\dom(f)$ and $(x_i,x_{i+1})\in\dom(f)$ for all $0\leq i<m$. 
\item Given $m\geq 1$, if $\xbar=(x_0,\ldots,x_m)$ is an $f$-sequence, then we let $f[\xbar]$ denote $f(x_0,x_1)+f(x_1,x_2)+\ldots+f(x_{n-1},x_m)$.
\item Given $m\geq 1$, $f$ is \textbf{$m$-transitive} if $f(x_0,x_m)\leq f[\xbar]$ for all $f$-sequences $\xbar=(x_0,\ldots,x_m)$.
\end{enumerate}
\end{definition}

By adapting the minimal path metric for graphs to the setting of metric spaces, we obtain the following test for consistency of partial semimetric spaces.  

\begin{lemma}\label{consistent}
A partial semimetric space $(X,f)$ is consistent if and only if $f$ is $m$-transitive for all $m\geq 1$. 
\end{lemma}
\begin{proof}
First, suppose $(X,f)$ is consistent and let $d$ be a pseudometric on $X$ extending $f$. If $m>0$ and $\xbar=(x_0,x_1,\ldots,x_m)$ is an $f$-sequence, then $d(x_0,x_m)=f(x_0,x_m)$ and $f[\xbar]=d(x_0,x_1)+\ldots+d(x_{m-1},x_m)$. Therefore $f(x_0,x_m)\leq f[\xbar]$ by repeated application of the triangle inequality. 

Conversely, suppose $f$ is $m$-transitive for all $m\geq 1$. Given $x,y\in X$, set
$$
d(x,y)=\inf\{f[\xbar]:\text{$\xbar=(x_0,\ldots,x_m)$ is an $f$-sequence with $x_0=x$ and $x_m=y$}\}
$$
Then $d$ is clearly a pseudometric on $X$, and so we only need to show that $d$ extends $f$. Indeed, if $(x,y)\in\dom(f)$ then $(x,y)$ is an $f$-sequence and so $d(x,y)\leq f(x,y)$. On the other hand, $f(x,y)\leq d(x,y)$ since $f$ is $m$-transitive for all $m\geq 1$.
\end{proof}

We now begin our analysis of indiscernible sequences in $\U$. 

\begin{lemma}\label{inds}
Suppose $(\abar^l)_{l<\omega}$ is an indiscernible sequence in $\U$, with $\ell(\abar^0)=k$ for some $k\geq 1$. Given $1\leq i,j\leq k$, define $\epsilon_{i,j}=d(a^0_i,a^1_j)$.
\begin{enumerate}[$(a)$]
\item For any $n\geq 1$ and $i_0,\ldots,i_n\in\{1,\ldots,k\}$, 
$$
\epsilon_{i_0,i_n}\leq\epsilon_{i_0,i_1}+\epsilon_{i_1,i_2}+\ldots+\epsilon_{i_{n-1},i_n}.
$$
\item\label{switch} For any $n\geq 1$ and $i_0,\ldots,i_n\in\{1,\ldots,k\}$, if there are $0\leq s<t\leq n$ such that $i_s=i_t$ then
$$
\epsilon_{i_n,i_0}\leq\epsilon_{i_0,i_1}+\epsilon_{i_1,i_2}+\ldots+\epsilon_{i_{n-1},i_n}.
$$
\end{enumerate}
\end{lemma}
\begin{proof}
Part $(a)$: By indiscernibility,
$$
 \epsilon_{i_0,i_n} = d(a^0_{i_0},a^n_{i_n}) \leq d(a^0_{i_0},a^1_{i_1})+\ldots+d(a^{n-1}_{i_{n-1}},a^n_{i_n})=\epsilon_{i_0,i_1}+\ldots+\epsilon_{i_{n-1},i_n}.
$$
Part $(b)$: We assume $0<s<t<n$ (the cases $s=0$ and $t=n$ are similar and left to the reader). By indiscernibility and part $(a)$, we have
\begin{align*}
\epsilon_{{i_n},{i_0}} &= d(a^1_{i_n},a^2_{i_0}) \\
 &\leq d(a^1_{i_n},a^0_{i_s})+d(a^0_{i_s},a^3_{i_s})+d(a^3_{i_s},a^2_{i_0}) \\
 &= \epsilon_{i_s,i_n}+\epsilon_{i_s,i_s}+\epsilon_{i_0,i_s}\\
 &= \epsilon_{i_0,i_s}+\epsilon_{i_s,i_t}+\epsilon_{i_t,i_n}\\
 &\leq \epsilon_{i_0,i_1}+\epsilon_{i_1,i_2}+\ldots+\epsilon_{i_{n-1},i_n}.\qedhere
\end{align*}
\end{proof}

These transitivity properties will be useful when trying to show that certain indiscernible sequences are $n$-cyclic for some $n$. Specifically, we now show that the task of proving an indiscernible sequence is $n$-cyclic reduces to just checking transitivity for sequences like those in the previous lemma.

\begin{definition} Let $(X,f)$ be a partial semimetric space.
\begin{enumerate}
\item If $\xbar=(x_0,\ldots,x_m)$ is a sequence of elements of $X$, and $0\leq i<j\leq m$, let
$$
\xbar[x_i,x_j]:=(x_i,x_{i+1},\ldots,x_j).
$$
\item If $\xbar=(x_0,\ldots,x_m)$ is a sequence of elements of $X$, then a \textbf{subsequence of $\xbar$} is a sequence of the form $(x_0,x_{i_1},\ldots,x_{i_k},x_m)$, for some $0< i_1<\ldots<i_k<m$. If $1\leq k\leq m-2$ then the subsequence is \textbf{proper}.
\end{enumerate}
\end{definition}

\begin{lemma}\label{reduction}
Suppose $(\abar^l)_{l<\omega}$ is an indiscernible sequence in $\U$, with $\ell(\abar^0)=k$ for some $k\geq 1$. Given $1\leq i,j\leq k$, define $\epsilon_{i,j}=d(a^0_i,a^1_j)$. Given $n\geq 2$, the following are equivalent:
\begin{enumerate}[$(i)$]
\item $(\abar^l)_{l<\omega}$ is $n$-cyclic;
\item for all $i_1,\ldots,i_n\in\{1,\ldots,k\}$, $\epsilon_{i_n,i_1}\leq\epsilon_{i_1,i_2}+\epsilon_{i_2,i_3}+\ldots+\epsilon_{i_{n-1},i_n}$.
\end{enumerate}
\end{lemma}
\begin{proof}
Fix $n\geq 2$ and an indiscernible sequence $(\abar^l)_{i<\omega}$, with $|\abar^0|=k$, for some $k\geq 1$. Let $X=\{x^l_i:1\leq l\leq n,~1\leq i\leq k\}$. We define a partial semimetric $f: X\times X\longrightarrow [0,1]$ as follows. 

Define $\dom(f)$ to be the symmetric closure of
$$
\{(x^l_i,x^m_j):1\leq i,j\leq k,~l,m<\omega,\text{ and }l\in\{m,m+1\}\text{ or }(l,m)=(1,n)\}.
$$
Given $(x^l_i,x^m_j)\in\dom(f)$, we define $f(x^l_i,x^m_j)=d(a^l_i,a^m_j)$ if $(l,m)\not\in\{(1,n),(n,1)\}$ and  $f(x^l_i,x^m_j)=d(a^1_i,a^0_j)$ if $(l,m)=(1,n)$. 

By Proposition \ref{embed}, $(\abar^l)_{l<\omega}$ is $n$-cyclic if and only if $(X,f)$ is consistent. Together with Lemma \ref{consistent}, we obtain
\begin{equation}
\textnormal{$(\abar^l)_{l<\omega}$ is $n$-cyclic if and only if $f$ is $m$-transitive for all $m\geq 1$.}\tag{$\dagger$}
\end{equation}
We now proceed with the proof of the result.

$(i)\Rightarrow(ii)$: If $(\abar^l)_{l<\omega}$ is $n$-cyclic then for all $i_1,\ldots,i_n\in\{1,\ldots,k\}$,  it follows from $(\dagger)$ that
$$
\epsilon_{i_n,i_1}=f(x^1_{i_1},x^n_{i_n})\leq f(x^1_{i_1},x^2_{i_2})+\ldots+f(x^{n-1}_{i_{n-1}},x^n_{i_n})=\epsilon_{i_1,i_2}+\ldots+\epsilon_{i_{n-1},i_n}.
$$

$(ii)\Rightarrow(i)$: Assume $(ii)$ holds. By $(\dagger)$, it suffices to prove, by induction on $m$, that $f$ is $m$-transitive for all $m\geq 1$. The case $m=1$ follows immediately by symmetry of $f$. For the induction step, fix $m>1$ and assume that $f$ is $j$-transitive for all $j<m$. Fix an $f$-sequence $\ubar=(u_0,\ldots,u_m)$. We want to show $f(u_0,u_m)\leq f[\ubar]$.\\

\noindent\textit{Claim}: If some proper subsequence of $\ubar$ is an $f$-sequence then $f(u_0,u_m)\leq f[\ubar]$.

\noindent\textit{Proof}: Let $\vbar=(v_0,\ldots,v_j)$ be a proper $f$-subsequence, where $j<m$, $v_0=u_0$, and $v_j=u_m$. For $0\leq t\leq j-1$, set $\ubar_t=\ubar[v_t,v_{t+1}]$. By induction, we have
$$
f(u_0,u_m)=f(v_0,v_j)\leq f(v_0,v_1)+\ldots+f(v_{j-1},v_j)\leq f[\ubar_0]+\ldots+f[\ubar_{j-1}]=f[\ubar].\text{\claim}
$$
~

Suppose $\ubar=(x^{e_0}_{i_0},\ldots,x^{e_m}_{i_m})$ for some $1\leq e_t\leq n$ and $1\leq i_t\leq k$. 

\noindent\textit{Case 1}: $e_s=e_t$ for some $s<t$. We will show that either $\ubar$ is isometric to a triangle in $(\abar^l)_{l<\omega}$, or $\ubar$ contains a proper $f$-subsequence, in which case we apply the claim.

First, if $m=2$ then $\ubar$ is a triangle with at least two points in $\xbar^{e_s}=\xbar^{e_t}$ and all three edges in $\dom(f)$. It follows from the definition of $\dom(f)$ that $\ubar$ is isometric to a triangle in $(\abar^l)_{l<\omega}$. Therefore $f(u_0,u_m)\leq f[\ubar]$ by the triangle inequality. So we assume $m>2$. In the rest of the cases, we find a proper $f$-subsequence of $\ubar$.

Suppose $s=0$ and $t=m$. Then $\vbar=(u_0,u_1,u_m)$ is a proper subsequence of $\ubar$, since $m>2$. Moreover, $\vbar$ is an $f$-sequence by definition of $\dom(f)$. So we may assume that $s=0$ implies $t<m$. 

If $s+1<t$ then, combined with the assumption that $s=0$ implies $t<m$, it follows that $\vbar=(u_0,\ldots,u_s,u_t,\ldots,u_m)$ is a proper subsequence of $\ubar$. Moreover, $\vbar$ is an $f$-sequence since $e_s=e_t$ implies $(u_s,u_t)\in\dom(f)$. So we may assume $t=s+1$.

If $t<m$ then $\vbar=(u_0,\ldots,u_s,u_{t+1},\ldots,u_m)$ is a proper subsequence of $\ubar$. Moreover, $\vbar$ is an $f$-sequence since $e_s=e_t$ implies $(u_s,u_{t+1})\in\dom(f)$. 

Finally, if $t=m$ then $\vbar=(u_0,\ldots,u_{m-2},u_m)$ is a proper subsequence of $\ubar$. Moreover, $\vbar$ is an $f$-sequence since $e_s=e_t$ implies $(u_{m-2},u_m)\in\dom(f)$.

\noindent\textit{Case 2}: $e_s\neq e_t$ for $s\neq t$. Since $\ubar$ is an $f$-sequence, it follows from the definition of $\dom(f)$ that $m=n-1$ and, moreover, there is a permutation $\sigma\in S_n$, which is some power of $(1~2~\ldots~n)$, such that either $(\sigma(e_0),\ldots,\sigma(e_m))=(1,\ldots,n)$ or $(\sigma(e_0),\ldots,\sigma(e_m))=(n,\ldots,1)$. Note that, if $\sigma_*:X\longrightarrow X$ is such that $\sigma_*(x^e_i)=x^{\sigma(e)}_i$, then for all $x,y\in X$, we have $f(x,y)=f(\sigma_*(x),\sigma_*(y))$. Therefore we may assume $(e_0,\ldots,e_m)$ is either $(1,\ldots,n)$ or $(n,\ldots,1)$.

Next, note that $f(u_0,u_m)\leq f[\ubar]$ if and only if $f(u_m,u_0)\leq f[(u_m,u_{m-1},\ldots,u_0)]$. Therefore we may assume $(e_0,\ldots,e_m)=(1,\ldots,n)$, and so $\ubar=(x^1_{i_0},\ldots,x^n_{i_{n-1}})$. By $(ii)$, we have
$$
f(x^1_{i_0},x^n_{i_{n-1}})=\epsilon_{i_{n-1},i_0}\leq\epsilon_{i_0,i_1}+\ldots+\epsilon_{i_{n-2},i_{n-1}}=f[\ubar],
$$
as desired.
\end{proof}

We are now ready prove the main result of this section, a corollary of which will be the desired classification of $\Th(\cU)$ in the $\SOP_n$ hierarchy.

\begin{theorem}\label{cyclicthm} Fix $n\geq 1$.
\begin{enumerate}[$(a)$] 
\item Any indiscernible sequence in $\U$, of tuples of length $n$, is $(n+1)$-cyclic.
\item There is an indiscernible sequence in $\U$, of tuples of length $n$, that is not $n$-cyclic.
\end{enumerate}
\end{theorem}
\begin{proof}
Part $(a)$: Suppose $(\abar^l)_{l<\omega}$ is an indiscernible sequence in $\U$, with $\ell(\abar^0)=n$. To show that $(\abar^l)_{l<\omega}$ is $(n+1)$-cyclic, we use the characterization of Lemma \ref{reduction}. If $i_0,\ldots,i_n\in\{1,\ldots,n\}$ then there are $0\leq s<t\leq n$ such that $i_s= i_t$. By Lemma \ref{inds}$(\ref{switch})$, we have $\epsilon_{i_n,i_0}\leq\epsilon_{i_0,i_1}+\ldots+\epsilon_{i_{n-1},i_n}$.

Part $(b)$: We construct $(\abar^l)_{l<\omega}$, where $\ell(\abar^l)=n$, and define distances as follows. Given $k\leq l<\omega$,
$$
d(a^k_i,a^l_j)=
\begin{cases}
\frac{j-i+1}{n}&\textnormal{ if $k<l$ and $i\leq j$, or $k=l$ and $i<j$}\\
\frac{i-j}{n} &\textnormal{ if $k<l$ and $i>j$.}
\end{cases}
$$
We leave it to the reader to verify that this sequence satisfies the triangle (the verification is entirely routine, but very tedious). Therefore, we may view $(\abar^l)_{l<\omega}$ as an indiscernible sequence in $\U$. It remains to show that this sequence is not $n$-cyclic. Let $p(\xbar,\ybar)=\tp(\abar^0,\abar^1)$ and suppose, toward a contradiction, that $p(\xbar^1,\xbar^2)\cup\ldots\cup p(\xbar^{n-1},\xbar^n)\cup p(\xbar^n,\xbar^1)$ is satisfied by some $(\cbar^1,\ldots,\cbar^n)$. Note that, for $1\leq i<n$, we have $d(a^0_{i+1},a^1_i)=\textstyle\frac{i+1-i}{n}=\frac{1}{n}$. Therefore, for $1\leq i<n$, we have $d(c^{n-i}_{i+1},c^{n-i+1}_i)=\frac{1}{n}$. This means
$$
d(c^1_n,c^n_1)\leq d(c^1_n,c^2_{n-1})+d(c^2_{n-1},c^3_{n-2})+\ldots+d(c^{n-1}_2,c^n_1)=\textstyle\frac{n-1}{n}.
$$
But this is a contradiction, since $d(a^0_1,a^1_n)=\frac{n-1+1}{n}=1$, and so we must have $``d(x^n_1,x^1_n)=1"\in p(\xbar^n,\xbar^1)$.
\end{proof}

Applying Proposition \ref{SOP}, we obtain the desired classification of $\Th(\cU)$ in the $\SOP_n$-hierarchy.

\begin{corollary}
$\Th(\cU)$ is $\NSOP_\infty$, and $\SOP_n$ for all $n\geq 3$.
\end{corollary}

As a final remark concerning dividing lines, we consider $\TP_2$. 

\begin{theorem}
$\Th(\cU)$ has $\TP_2$.
\end{theorem}
\begin{proof}
We define the array $(a_{i,j})_{i,j<\omega}$ such that 
$$
d(a_{m,i},a_{n,j})=
\begin{cases} 
1 &\textrm{ if } m=n,~i\neq j\\ 
\frac{2}{3} &\textrm{ if }m\neq n
\end{cases}
$$
Note that any nontrivial triangle in this array has sides with distances from $\{\frac{2}{3},1\}$. Therefore the array satisfies the triangle inequality. If $\sigma\in\omega^\omega$ then the distance between any two distinct elements of $\{a_{n,\sigma(n)}:n<\omega\}$ is $\frac{2}{3}$. Therefore $\{d(x,a_{n,\sigma(n)})=\frac{1}{3}:n<\omega\}$ is satisfiable. On the other hand, if $n<\omega$ and $i<j<\omega$, then $d(a_{n,i},a_{n,j})=1$. Therefore, for any $n<\omega$, $\{d(x,a_{n,i})=\frac{1}{3}:i<\omega\}$ is $2$-unsatisfiable.
\end{proof}

\subsection{Dividing in $\Th(\cU)$}\label{divchar}

We now turn to the question of forking and dividing in $\Th(\cU)$. In this section, we characterize dividing for complete types. We will later show that this characterization satisfies condition $(iii)$ of Theorem \ref{ext reduct}. As a result, forking and dividing are the same for complete types, and we will have given a purely combinatorial characterization of both notions of independence. 

Toward the characterization of dividing in $\Th(\cU)$, we begin with the following strengthening of finite character (Fact \ref{basic}).

\begin{lemma}\label{to one}
 Given $A,B,C\subset\U$, $A\textstyle\ind^d_C B$ if and only if $a\ind^d_C b_1b_2$ for all $a\in A$ and $b_1,b_2\in B$.
\end{lemma}
\begin{proof}
$(\Rightarrow)$: Follows from Fact \ref{basic}.

$(\Leftarrow$): Let $\abar$ enumerate $A$ and $\bbar$ enumerate $B$. Suppose $\tp(\abar/BC)$ divides over $C$. Then there is a $C$-indiscernible sequence $(\bbar^l)_{l\in \omega}$, with $\bbar^0=\bbar$, such that if $p(\xbar,\ybar) = \tp(\abar,\bbar/C)$, then $\bigcup_{l< \omega} p(\xbar,\bbar^{l})$ is unsatisfiable. In other words, the following is unsatisfiable:
\begin{multline*}
\{d(x_{i}, x_{j})=d(a_{i}, a_{j}):a_i,a_j\in\abar\} \cup \{d(x_i,c)=d(a_i,c):a_i\in\abar,~c\in C\}\\
\cup \{d(x_{i}, b_{j}^{l}) = d(a_{i}, b_{j}): a_i\in\abar,~b_i\in\bbar,~l< \omega \} \cup \tp((\bbar^l)_{l<\omega}/C).
\end{multline*}

So there is a failure of the triangle inequality among three points in this type. By indiscernibility, the only possible failures are between three points of the form $\{x_i,b^l_j,b^m_k\}$ for some $a_i\in\abar$ and $b_j,b_k\in\bbar$. In this case, if $q(x,b_j,b_i)=\tp(a_i,b_j,b_k/C)$ then $\bigcup_{n<\omega}q(x,b^n_j,b^n_k)$ is unsatisfiable. Therefore $a_i\nind^d_C b_jb_k$, as desired.
\end{proof}

Next, we show explicitly how indiscernible sequences control dividing.

\begin{definition}
Fix $b_1,b_2\in\U$ and $C\subset\U$. Let $\bbar=(b_1,b_2)$. Define the sets
\begin{align*}
\IS(b_1,b_2/C) &= \{(\bbar^l)_{l<\omega}:\text{$\bbar^0\equiv_C\bbar$ and $(\bbar^l)_{l<\omega}$ is $C$-indiscernible}\},\\
\Gamma(b_1,b_2/C) &= \{d(b^0_1,b^1_2):(\bbar^l)_{l<\omega}\in\IS(b_1,b_2/C)\}.
\end{align*}
\end{definition}

\begin{proposition}\label{ind commute}
If $b_1,b_2\in\U$ and $C\subset\U$ then $\Gamma(b_1,b_2/C)=\Gamma(b_2,b_1/C)$.
\end{proposition}
\begin{proof}
It suffices to show $\Gamma(b_2,b_1/C)\seq\Gamma(b_1,b_2/C)$. Suppose $(\bbar^l)_{l<\omega}\in\IS(b_2,b_1/C)$. We want to show $d(b^0_2,b^1_1)\in\Gamma(b_1,b_2/C)$. Let $\omega^*=\{l^*:l<\omega\}$, ordered so that $l^*>(l+1)^*$. By compactness we may assume the sequence is indexed $(\bbar^l)_{l\in I}$, where $I=\omega+\omega^*$. Define the sequence $(\abar^l)_{l<\omega}$ such that $\abar^l=(b^{l^*}_1,b^{l^*}_2)$. Then $(\abar^l)_{l<\omega}$ is $C$-indiscernible and $\abar^0\equiv_C(b_1,b_2)$. Therefore $(\abar^l)_{l<\omega}\in\IS(b_1,b_2/C)$ and so
$$
d(b^0_2,b^1_1)=d(b^1_1,b^0_2)=d(b^{0^*}_1,b^{1^*}_2)=d(a^0_1,a^1_2)\in\Gamma(b_1,b_2/C),
$$
as desired.
\end{proof}

\begin{lemma}\label{lemma4}
 Suppose $a,b_1,b_2\in\U$ and $C\subset\U$. Then $a\ind^d_C b_1b_2$ if and only if for all $i,j\in\{1,2\}$,
 $$
 d(a,b_i)+d(a,b_j)\geq \sup\Gamma(b_i,b_j/C) \text{ and } |d(a,b_i)-d(a,b_j)|\leq \inf\Gamma(b_i,b_j/C).
 $$
 \end{lemma}
 \begin{proof}
 $(\Leftarrow)$:   Suppose $a\nind^d_C \bbar$, where $\bbar=(b_1,b_2)$.  Let $p(x,\ybar)$ denote $\tp(a,\bbar/C)$.  Then there is an indiscernible sequence $(\bbar^l)_{l< \omega}\in\IS(b_1,b_2/C)$ such that $\bigcup_{l<\omega} p(x, \bbar^l)$ is unsatisfiable, and therefore contains some failure of the triangle inequality. The possible failures are:
 \begin{enumerate}
 \item $d(b^l_i, b^m_j)>d(a,b_i)+d(a,b_j)$ for some $i,j\in\{1,2\}$ and $l,m<\omega$;
\item $d(a,b_j)>d(a,b_i)+d(b^l_i,b^m_j)$ for some $i,j\in\{1,2\}$ and $l,m<\omega$.
\end{enumerate}
In either case, since $\bbar^l\equiv_C\bbar\equiv_C\bbar^m$, it follows that $l\neq m$.  By Proposition \ref{ind commute}, we have $d(b^l_i,b^m_j)\in\Gamma(b_i,b_j/C)$. Therefore $\inf\Gamma(b_i,b_j/C)\leq d(b^l_i,b^m_j)\leq \sup\Gamma(b_i,b_j/C)$. If (1) holds then $d(a,b_i) + d(a,b_j)<d(b_i^l, b_j^m)\leq \sup\Gamma(b_i,b_j/C)$. If (2) holds then $|d(a,b_j) - d(a,b_i)|>d(b_i^l,b_j^m)\geq \inf\Gamma(b_i,b_j/C)$.

$(\Rightarrow)$: We again set $\bbar=(b_1,b_2)$ and $p(x,\ybar)=\tp(a,\bbar/C)$. First suppose there are $i,j\in\{1,2\}$ such that $d(a,b_i) + d(a,b_j) < \sup\Gamma(b_i,b_j/C)$. By definition, there is $(\bbar^l)_{l< \omega}\in\IS(b_i,b_j/C)$ and some $r>d(a,b_i)+d(a,b_j)$ such that for all $l<m<\omega$, $d(b_i^l,b_j^m)=r$.  Then
$$
\{d(x,b^0_i) = d(a,b_i),~d(x,b^1_j) = d(a,b_j),~d(b_i^0,b^j_1)=r)\}\seq\bigcup_{l<\omega}p(x,\bbar^l),
$$
and so $\bigcup_{l<\omega}p(x,\bbar^l)$ is unsatisfiable. Therefore $a\nind^d_C b_ib_j$.

Finally, suppose there are $i,j\in\{1,2\}$ such that $|d(a,b_i) - d(a,b_j)| > \inf\Gamma(b_i,b_j/C)$. By definition, there is $(\bbar^l)_{l< \omega}\in\IS(b_i,b_j/C)$ and some $r<|d(a,b_i)-d(a,b_j)|$ such that for all $l<m<\omega$, $d(b_i^l,b_j^m)=r$.  Once again, this implies $\bigcup_{l<\omega}p(x,\bbar^l)$ is unsatisfiable, and so $a\nind^d_C b_ib_j$.
\end{proof}

By the previous results, a complete characterization of dividing in the Urysohn sphere rests on understanding the possible values for $d(b^0_1,b^1_2)$ in sequences $(\bbar^l)_{l<\omega}\in\IS(b_1,b_2/C)$. To this end, we define the following distance calculations.

\begin{definition}
Fix $b_1,b_2\in\U$ and $C\subset\U$. Define the values
\begin{align*}
d_{\max}(b_1,b_2/C) &= \inf_{c\in C}(d(b_1,c)+d(b_2,c)),\\
d_{\min}(b_1,b_2/C) &= \max\left\{\sup_{c\in C}|d(b_1,c)-d(b_2,c)|,\textstyle\frac{1}{3}\displaystyle d(b_1,b_2)\right\}.
\end{align*}
\end{definition}

The following properties of these notions will be extremely useful.

\begin{lemma}\label{forkprops} Fix $b_1,b_2,b_3\in \U$ and $C\subset\U$. 
\begin{enumerate}[$(a)$]
\item $d_{\max}(b_1,b_2/C)\leq d_{\max}(b_1,b_3/C)+\sup_{c\in C}|d(b_2,c)-d(b_3,c)|$.
\item $d_{\min}(b_1,b_2/C)\leq d_{\min}(b_1,b_3/C)+d_{\min}(b_2,b_3/C)$.
\end{enumerate}
\end{lemma}
\begin{proof}
Part $(a)$: For any $c'\in C$, we have
\begin{align*}
d_{\max}(b_1,b_2/C) &\leq d(b_1,c')+d(b_2,c')\\
 &\leq d(b_1,c')+d(b_3,c')+|d(b_2,c')-d(b_3,c')|\\
 &\leq d(b_1,c')+d(b_3,c')+ \sup_{c\in C}|d(b_2,c)-d(b_3,c)|.
 \end{align*}
Therefore $d_{\max}(b_1,b_2/C)\leq d_{\max}(b_1,b_3/C)+\sup_{c\in C}|d(b_2,c)-d(b_3,c)|$.

Part $(b)$: If $d_{\min}(b_1,b_2/C)=\frac{1}{3}d(b_1,b_2)$ then the result is clear. So we may assume $d_{\min}(b_1,b_2/C)=\sup_{c\in C}|d(b_1,c)-d(b_2,c)|$.  Given any $c\in C$, we have $|d(b_1,c)-d(b_2,c)|\leq |d(b_1,c)-d(b_3,c)|+|d(b_2,c)-d(b_3,c)|$. Therefore
\begin{align*}
d_{\min}(b_1,b_2/C)&\leq \sup_{c\in C}|d(b_1,c)-d(b_3,c)|+\sup_{c\in C}|d(b_2,c)-d(b_3,c)|\\
 &\leq d_{\min}(b_1,b_3/C)+d_{\min}(b_2,b_3/C).\qedhere
\end{align*}
\end{proof}

Note that for any $b_1,b_2\in\U$ and $C\subset\U$,
$$
d_{\min}(b_1,b_2/C)\leq d(b_1,b_2)\leq d_{\max}(b_1,b_2/C).
$$
We show that $\Gamma(b_1,b_2/C)$ is determined by these values.

\begin{lemma}\label{min ind}
For any $b_1,b_2\in\U$ and $C\subset\U$,
$$
\Gamma(b_1,b_2/C)= [d_{\min}(b_1,b_2/C),d_{\max}(b_1,b_2/C)].
$$
\end{lemma}
\begin{proof}
We first show $\Gamma(b_1,b_2/C)\seq [d_{\min}(b_1,b_2/C),d_{\max}(b_1,b_2/C)]$. Fix $(\bbar^l)_{l<\omega}\in\Gamma(b_1,b_2/C)$ and set $\gamma=d(b^0_1,b^1_2)$. 
For any $c\in C$, the three points $\{b^0_1,b^1_2,c\}$ imply 
$$
|d(b_1,c)-d(b_2,c)|\leq \gamma\leq d(b_1,c)+d(b_2,c)
$$
We also have
$$
d(b_1,b_2)=d(b^1_1,b^1_2)\leq d(b^1_1,b^2_2)+d(b^0_1,b^2_2)+d(b^0_1,b^1_2)= 3\gamma.
$$
Altogether, $d_{\min}(b_1,b_2/C)\leq\gamma\leq d_{\max}(b_1,b_2/C)$, as desired.

Conversely, fix $\gamma\in[d_{\min}(b_1,b_2/C),d_{\max}(b_1,b_2/C)]$. We define the sequence $(\bbar^l)_{l<\omega}$ such that, for $i,j\in\{1,2\}$ and $l\leq m<\omega$,
$$
d(b^l_i,b^m_j)=\begin{cases}
d(b_i,b_j) & \text{if $l=m$}\\
\min\{d_{\max}(b_i,b_i/C),d(b_1,b_2)+\gamma,2\gamma\} & \text{if $l<m$, $i=j$}\\
\gamma & \text{if $l<m$, $i\neq j$.}
\end{cases}
$$
We will show that the definition of this sequence satisfies the triangle inequality.  From this, it will then follow that $(\bbar^l)_{l<\omega}$ is an indiscernible sequence in $\U$, which witnesses $\gamma\in\Gamma(b_1,b_2/C)$.

By indiscernibility in the definition of $(\bbar^l)_{l<\omega}$, the nontrivial triangles to check are those with the following vertex sets:
\begin{enumerate}
\item $\{b^l_i,b^m_j,c\}$ for some $i,j\in\{1,2\}$, $l<m<\omega$, and $c\in C$,
\item $\{b^l_i,b^m_j,b^n_k\}$ for some $i,j,k\in\{1,2\}$ and $l\leq m\leq n<\omega$.
\end{enumerate}

\noindent\textit{Case $1$}: $\{b^l_i,b^m_j,c\}$ for some $i,j\in\{1,2\}$, $l<m<\omega$, and $c\in C$. We need to show
$$
|d(b_i,c)-d(b_j,c)|\leq d(b^l_i,b^m_j)\leq d(b_i,c)+d(b_j,c).
$$
\begin{enumerate}[$(i)$]
\item $d(b^l_i,b^m_j)\leq d(b_i,c)+d(b_j,c)$. 

In all cases we have $d(b^l_i,b^m_j)\leq d_{\max}(b_i,b_j/C)\leq d(b_i,c)+d(b_j,c)$.

\item $d(b^l_i,b^m_j)\geq|d(b_i,c)-d(b_j,c)|$. 

We may clearly assume $i\neq j$. Then $d(b^l_i,b^m_j)=\gamma\geq|d(b_i,c)-d(b_j,c)|$.
\end{enumerate}

\noindent\textit{Case $2$}: $\{b^l_i,b^m_j,b^n_k\}$ for some $i,j,k\in\{1,2\}$ and $l\leq m\leq n<\omega$. Note that $i,j,k$ cannot all be distinct. By indiscernibility, and the symmetry in the definition of $(\bbar^l)_{l<\omega}$, we may assume $l<n=m$ or $l<m<n$.

\textit{Subcase 2.1}: $l<m=n$. Then we may assume $j\neq k$; and it suffices to  check the following two inequalities.

\begin{enumerate}[$(i)$]
\item $d(b_1,b_2)\leq d(b^l_i,b^m_j)+d(b^l_i,b^m_k)$. 

Without loss of generality, we may assume $i=j$. We want to show
$$
d(b_1,b_2)\leq d(b^l_i,b^m_i)+\gamma.
$$
If $d(b^l_i,b^m_i)=d(b_1,b_2)+\gamma$ then this is trivial. If $d(b^l_i,b^m_i)=2\gamma$ then this is true since $d(b_1,b_2)\leq 3\gamma$. Suppose $d(b^l_i,b^m_i)=d_{\max}(b_i,b_i/C)$. Then, using Lemma \ref{forkprops}$(a)$,
$$
d(b_1,b_2)\leq d_{\max}(b_1,b_2/C)\leq d_{\max}(b_i,b_i/C)+d_{\min}(b_1,b_2/C)\leq d(b^l_i,b^m_i)+\gamma.
$$

\item $d(b^l_i,b^m_j)\leq d(b_1,b_2)+d(b^l_i,b^m_k)$.

Suppose $i=j$. Then $i\neq k$ so $d(b^l_i,b^m_j)\leq d(b_1,b_2)+\gamma=d(b_1,b_2)+d(b^l_i,b^m_k)$.

Suppose $i=k$. Then $i\neq j$ so $d(b^l_i,b^m_j)=\gamma$. If $d(b^l_i,b^m_k)=d(b_1,b_2)+\gamma$ or $d(b^l_i,b^m_k)=2\gamma$ then the inequality is obvious. So we may assume $d(b^l_i,b^m_k)=d_{\max}(b_i,b_i/C)$. Then, using Lemma \ref{forkprops}$(a)$,
$$
d(b^l_i,b^m_j)=\gamma\leq d_{\max}(b_1,b_2/C)\leq d(b_1,b_2)+d_{\max}(b_i,b_i/C)=d(b_1,b_2)+d(b^l_i,b^m_k).
$$
\end{enumerate}

\textit{Subcase 2.2}: $l<m<n$.  By indiscernibility, it suffices to check
$$
d(b^l_i,b^m_j)\leq d(b^l_i,b^n_k)+d(b^m_j,b^n_k).
$$
If $i\neq j$ or $i=j=k$ then the inequality is trivial. So assume $i=j\neq k$. Then we have $d(b^l_i,b^m_j)\leq 2\gamma=d(b^l_i,b^n_k)+d(b^m_j,b^n_k)$.
\end{proof}

Using this result we can reformulate Lemma \ref{lemma4} as follows.

\begin{corollary}\label{divcor}
 Suppose $a,b_1,b_2\in\U$ and $C\subset\U$. Then $a\ind^d_C b_1b_2$ if and only if for all $i,j\in\{1,2\}$,
 $$
 d(a,b_i)+d(a,b_j)\geq d_{\max}(b_i,b_j/C) \text{ and } |d(a,b_i)-d(a,b_j)|\leq d_{\min}(b_i,b_j/C).
 $$
 \end{corollary}

If we combine Lemma \ref{to one} and Corollary \ref{divcor} then we obtain the complete characterization of dividing independence in the Urysohn sphere.

\begin{theorem}\label{dividethm}
If $A,B,C\subset\U$ then $A\ind^d_C B$ if and only if for all $b_1,b_2\in B$,
$$
d_{\max}(b_1,b_2/AC)=d_{\max}(b_1,b_2/C)\text{ and } d_{\min}(b_1,b_2/AC)=d_{\min}(b_1,b_2/C).
$$
\end{theorem}

\subsection{Extension for dividing independence}\label{subsec:ext}

The next goal is to show our characterization of $\ind^d$ satisfies condition $(ii)$ of Theorem \ref{ext reduct}, which we restate as follows.

\begin{theorem}\label{finite-one}
Fix $B,C\subset\U$ and a singleton $b_*\in\U$. For any $A\subset\U$, if $A\ind^d_C B$ then there is $A'\equiv_C A$ such that $A'\ind^d_C Bb_*$. 
\end{theorem}

Toward the proof of Theorem \ref{finite-one}, fix $B,C\subset\U$ and $b_*\in\U$. We may assume $C\seq B$. Given $b\in B$, let $\delta_b=d_{\min}(b_*,b/C)$ and $\epsilon_b=d_{\max}(b_*,b/C)$.

\begin{definition}\label{LU} Given $a\in\U$, define
\begin{alignat*}{3}
U(a) &= U(a,b_*/B,C)&&= \inf_{b\in B}(d(a,b)+\delta_b)\\
L(a) &= L(a,b_*/B,C)&&= \sup_{b\in B}\max\{\epsilon_b\dotminus d(a,b),d(a,b)\dotminus\delta_b\}
\end{alignat*}
\end{definition}

We first prove two lemmas. 

\begin{lemma}\label{1tpext}  Fix $a\in\U$.
\begin{enumerate}[$(a)$]
\item Assume $\gamma\in[0,1]$ is such that $L(a)\leq\gamma\leq U(a)$ and $d_{\max}(b_*,b_*/C)\leq\gamma+\gamma$. If $a'\in\U$ is such that $a'\equiv_B a$ and $d(a',b_*)=\gamma$, then $a'\ind^d_C Bb_*$.  
\item If $a\ind^d_C B$ then $L(a)\leq U(a)$ and $d_{\max}(b_*,b_*/C)\leq U(a)+U(a)$.
\item If $a\ind^d_C B$ and $a'\in\U$, with $a'\equiv_B a$ and $d(a',b_*)=U(a)$, then $a'\ind^d_C Bb_*$.
\end{enumerate}
\end{lemma}
\begin{proof}
Part $(a)$: Suppose $a'\in\U$ is such that $a'\equiv_B a$ and $d(a,b_*)=\gamma$. We use Theorem \ref{dividethm} to prove $a'\ind^d_C Bb_*$. First, note that $a'\equiv_B a$, and $L(a)\leq\gamma\leq U(a)$ together imply that for all $b\in B$,
\begin{align*}
d_{\max}(b,b_*/C) &=\epsilon_b\leq d(a',b_*)+d(a,b)=d(a',b_*)+d(a',b),\text{ and}\\
d_{\min}(b,b_*/C) &=\delta_b\geq |d(a',b_*)-d(a,b)|=|d(a',b_*)-d(a',b)|.
\end{align*}
Finally, we trivially have $|d(a',b_*)-d(a',b_*)|\leq d_{\min}(b_*,b_*/C)$; and, by assumption,
$$
d_{\max}(b_*,b_*/C)\leq\gamma+\gamma=d(a',b_*)+d(a',b_*).
$$

Part $(b)$: To show $L(a)\leq U(a)$, we fix $\alpha\in\{\epsilon_b\dotminus d(a,b):b\in B\}\cup\{d(a,b)\dotminus\delta_b:b\in B\}$ and $\beta\in\{d(a,b)+\delta_b:b\in B\}$, and show $\alpha\leq\beta$. Fix $b\in B$ such that $\beta=d(a,b)+\delta_b$.

\noindent\textit{Case 1}: $\alpha=\epsilon_{b'}\dotminus d(a,b')$ for some $b'\in B$. By $(\dagger)$ and Lemma \ref{forkprops}$(a)$, we have 
$$
\epsilon_{b'}\leq d_{\max}(b,b'/C)+\delta_b\leq d(a,b')+d(a,b)+\delta_b=d(a,b')+\beta.
$$

\noindent\textit{Case 2}: $\alpha=d(a,b')\dotminus \delta_{b'}$ for some $b'\in B$. By $(\dagger\dagger)$ and Lemma \ref{forkprops}$(b)$, we have
\begin{equation*}
d(a,b')\leq d_{\min}(b,b'/C)+d(a,b)\leq \delta_{b'}+\delta_b+d(a,b)=\delta_{b'}+\beta.
\end{equation*}

Next, we show $d_{\max}(b_*,b_*/C)\leq U(a)+U(a)$. Fix $b,b'\in B$. We want to show
$$
d_{\max}(b_*,b_*/C) \leq d(a,b) + 
\delta_b + d(a,b') + \delta_{b'}.
$$
By $(\dagger)$ and Lemma \ref{forkprops}$(a)$, we have 
\begin{align*}
d_{\max}(b_*,b_*/C) &\leq \epsilon_b+\delta_{b'}\\
 &\leq d_{\max}(b,b'/C)+\delta_{b}+\delta_{b'}\\
 &\leq d(a,b)+d(a,b')+\delta_b+\delta_{b'}.
\end{align*}

Part $(c)$. Follows immediately from $(a)$ and $(b)$.
\end{proof}

\begin{lemma}\label{Ulem}$~$
\begin{enumerate}[$(a)$]
\item If $a\in\U$ then $\sup_{b\in B}|d(a,b)-d(b_*,b)|\leq L(a)$ and $U(a)\leq d_{\max}(a,b_*/B)$.
\item If $a_1,a_2\in\U$ and $a_1a_2\ind^d_C B$ then $|U(a_1)-U(a_2)|\leq d(a_1,a_2)\leq U(a_1)+U(a_2)$.
\end{enumerate}
\end{lemma}
\begin{proof}
Part $(a)$: First, note that, for any $b\in B$, we have $U(a)\leq d(a,b)+\delta_b\leq d(a,b)+d(b_*,b)$. Next, for any $b\in B$, if $d(a,b)\leq d(b_*,b)$ then, since $d(b_*,b)\leq\epsilon_b$, we have $|d(b_*,b)-d(a,b)|\leq\epsilon_b-d(a,b)\leq L(a)$. On the other hand, if $d(b_*,b)<d(a,b)$ then, since $\delta_b\leq d(b_*,b)$, we have $|d(b_*,b)-d(a,b)|\leq d(a,b)-\delta_b\leq L(a)$. 

Part $(b)$: We first show $d(a_1,a_2)\leq U(a_1)+U(a_2)$. Fix $b,b'\in B$. Using $a_1\ind^d_C B$ and Lemma \ref{forkprops}$(b)$, we have
$$
|d(a_1,b)-d(a_1,b')|\leq d_{\min}(b,b'/C)\leq\delta_b+\delta_{b'}.
$$
 Therefore, $d(a_1,b')\leq d(a_1,b)+\delta_b+\delta_{b'}$, and so
$$
d(a_1,a_2)\leq d(a_1,b')+d(a_2,b')\leq d(a_1,b)+\delta_b+d(a_2,b')+\delta_{b'},
$$
as desired.

Next, we show $|U(a_1)- U(a_2)|\leq d(a_1,a_2)$. Without loss of generality, we may assume $U(a_1)\leq U(a_2)$ and show $U(a_2)\leq d(a_1,a_2) + U(a_1)$. For this, fix $b\in B$, and note that $U(a_2)\leq d(a_2,b)+\delta_b\leq d(a_1,a_2)+d(a_1,b)+\delta_b$.
\end{proof}

We are now ready to prove Theorem \ref{finite-one}.

\begin{proof}[Proof of Theorem \ref{finite-one}] Suppose $A\ind^d_C B$. We again assume, without loss of generality, that $C\seq B$. Let $\xbar=(x_a)_{a\in A}$ and define the type
$$
p(\xbar) := \tp_{\xbar}(A/B) \cup \{ d(x_a,b_*) = U(a) : a\in A\}.
$$
Suppose $A'$ realizes $p(\xbar)$. We immediately have $A'\equiv_B A$. Moreover, by Lemma \ref{1tpext}$(c)$, we have $a'\ind^d_C Bb_*$ for all $a'\in A$, and so $A'\ind^d_C Bb_*$ by Lemma \ref{to one}. Altogether, it suffices to show $p(\xbar)$ is satisfiable, which, as usual, means verifying triangle inequalities. 

We consider $\xbar Bb_*$ as a semimetric space with distances determined by $p(\xbar)$. Fix three points $V=\{v_1,v_2,v_3\}\seq\xbar Bb_*$. If $V\subseteq \xbar B$, or if $V \subseteq  Bb_*$, then $V$ is isometric to a triangle in $\U$. Thus we may assume one of the following holds:
\begin{enumerate}
\item $V=\{x_a,b,b_*\}$ for some $a\in A$ and $b\in B$;
\item $V=\{x_{a_1},x_{a_2},b_*\}$ for some $a_1,a_2\in A$.
\end{enumerate}

If $(1)$ holds then the distances in $V$ are $d(a,b)$, $d(b,b_*)$ and $U(a)$. Therefore, the triangle inequalities for $V$ follow from Lemma \ref{1tpext}$(b)$ and Lemma \ref{Ulem}$(a)$.

If $(2)$ holds then the distances in $V$ are $d(a_1,a_2)$, $U(a_1)$, and $U(a_2)$. Therefore, the triangle inequalities for $V$ follow from Lemma \ref{Ulem}$(b)$. 
\end{proof}

Combining Theorem \ref{ext reduct}, Theorem \ref{dividethm}, and Theorem \ref{finite-one}, we have completed the full characterization of forking and dividing for complete types in the Urysohn sphere.

\begin{theorem}
If $A,B,C\subset\U$ then $A\ind^f_C B$ if and only if $A\ind^d_C B$ if and only if for all $b_1,b_2\in B$,
$$
d_{\max}(b_1,b_2/AC)=d_{\max}(b_1,b_2/C)\text{ and }d_{\min}(b_1,b_2/AC)=d_{\min}(b_1,b_2/C).
$$
\end{theorem}

For future use, we also record the following consequence of our work.

\begin{corollary}\label{extop}
Fix $C\seq B\subset\U$ and $a,b_*\in\U$, with $a\ind^f_C B$. Given $\gamma\in[0,1]$, the following are equivalent:
\begin{enumerate}[$(i)$]
\item there is $a'\in\U$ such that $a'\equiv_B a$, $d(a',b_*)=\gamma$, and $a'\ind^f_C Bb_*$;
\item $L(a,b_*/B,C)\leq\gamma\leq U(a,b_*/B,C)$ and $d_{\max}(b_*,b_*/C)\leq \gamma+\gamma$.
\end{enumerate}
\end{corollary}

\subsection{Stationary types}\label{stat}

In this section, we use the characterization of nonforking in $\Th(\cU)$ to show that a type in $\Th(\cU)$ is stationary if and only if it is algebraic.

\begin{definition}
Let $C\subset\M\models T$ and $p\in S_n(C)$. Then $p$ is \textbf{stationary} if for all $B\supseteq C$, there is a unique nonforking extension of $p$ to $S_n(B)$.
\end{definition}

As with many other notions around nonforking, the study of stationary types began in stable theories. One important fact is that if $T$ is stable and $M\models T$ then any type over $M$ is stationary. Therefore, when extending nonforking types over a model, there is a unique choice of extension.

\begin{example}\label{existEx}
Suppose $T$ is a theory in which $\ind^f$ satisfies \emph{existence}, that is, $A\ind^f_C C$ for all $A$ and $C$. Then the most trivial kind of stationary type is one of the form $\tp(\abar/C)$ where $\abar\in\dcl(C)$. In this case, the type is stationary because there is a unique extension to any larger set (which is a nonforking extension by the assumption on $T$). Note that conversely, if a type has a unique extension to any larger parameter set then it must be of the form $\tp(\abar/C)$, with $\abar\in\dcl(C)$.
\end{example}

It follows easily from the definition of dividing that $\ind^d$ satisfies existence in any theory. Therefore, the previous setting applies to any theory in which $\ind^f$ and $\ind^d$ coincide. In particular, it applies to $\Th(\cU)$. The main result of this section is that, in $\Th(\cU)$, the only stationary types are the trivial kind in the previous example. In other words, if a type has more than one extension to any larger parameter set then it has more than one nonforking extension to any larger parameter set. We will need the following fact, which can be found in \cite{EaGo}.

\begin{fact}\label{closure}
If $C\subset\U$ then $\acl(C)=\dcl(C)=\overline{C}$, the usual metric space closure.
\end{fact}

Toward characterizing stationary types, we first consider extensions of $1$-types to larger parameter sets obtained by adding a single element.

\begin{theorem}\label{unique ext} Let $C\subset\U$ and $a,b\in\U$. The following are equivalent:
\begin{enumerate}[$(i)$]
\item $\tp(a/C)$ has a unique nonforking extension to $S_1(Cb)$;
\item $\tp(a/C)$ has a unique extension to $S_1(Cb)$;
\item $d_{\max}(a,b/C)=\sup_{c\in C}|d(a,c)-d(b,c)|$.
\end{enumerate}
\end{theorem}
\begin{proof}
$(i)\Rightarrow(iii)$: Let $d^*_{\min}(a,b/C)=\sup_{c\in C}|d(a,c)-d(b,c)|$, and note that $d^*_{\min}(a,b/C)\leq d_{\max}(a,b/C)$. Suppose $(ii)$ fails. Then we have $d^*_{\min}(a,b/C)<d_{\max}(a,b/C)$. Let $\gamma_0$ be the average of $d^*_{\min}(a,b/C)$ and $d_{\max}(a,b/C)$. In particular, we have $\gamma_0+\gamma_0=d_{\max}(a,b/C)+d^*_{\min}(a,b/C)$.

Next, let $U(a,b)=U(a,b/C,C)$ and $L(a,b)=L(a,b/C,C)$. Then it is straightforward to calculate that $U(a,b)=d_{\max}(a,b/C)$ and $L(a,b)=d^*_{\min}(a,b/C)$, and so $L(a,b)<\gamma_0<U(a,b)$. Moreover, by Lemma \ref{forkprops}$(a)$, we have
$$
d_{\max}(b,b/C)\leq d_{\max}(a,b/C)+d^*_{\min}(a,b/C)=\gamma_0+\gamma_0,
$$
It follows that, for any $\gamma\in[0,1]$, if $\gamma_0\leq\gamma\leq U(a,b)$ then, by Corollary \ref{extop}, there is $a'\equiv_C a$ such that $a'\ind^f_C b$ and $d(a',b)=\gamma$. Since $\gamma_0<U(a,b)$, there are infinitely many nonforking extensions of $\tp(a/C)$ to $S_1(Cb)$, and so $(i)$ fails.

$(iii)\Rightarrow(ii)$: Suppose $a'$ realizes an extension of $\tp(a/C)$ to $Cb$. Then $\tp(a'/Cb)$ is completely determined by the choice of $d(a',b)$. By $(iii)$ (and the triangle inequality), there is a unique choice.

$(ii)\Rightarrow(i)$: Immediate, since $\ind^f$ satisfies existence.
\end{proof}

\begin{corollary}\label{in closure}
Let $C\subset\U$ and $a\in\U$. Then $\tp(a/C)$ is stationary if and only if $a\in\overline{C}$.
\end{corollary}
\begin{proof}
The reverse direction is by Fact \ref{closure}. For the forward direction, if $\tp(a/C)$ is stationary then, by Theorem \ref{unique ext}, $d_{\max}(a,a/C)=0$. Therefore, there is a sequence $(c_n)_{n=0}^\infty$ in $C$ such that $(d(a,c_n)+d(a,c_n))\rightarrow 0$. It follows that $c\in\overline{C}$.
\end{proof}

Next, we show that stationarity of $n$-types is determined by stationarity of 1-types.

\begin{proposition}\label{one type} Let $C\seq B\subset\U$. The following are equivalent.
\begin{enumerate}[$(i)$]
\item $p\in S_n(C)$ has a unique nonforking extension to $S_n(B)$;
\item for all $1\leq i\leq n$ and $b\in B$, $p|_{x_i}$ has a unique nonforking extension to $S_1(Cb)$.
\end{enumerate}
\end{proposition}
\begin{proof} $(i)\Rightarrow(ii)$: Suppose for some $1\leq i\leq n$ and $b\in B$, $p|_{x_i}$ has two distinct nonforking extensions $q$ and $q'$ to $S_1(Cb)$. Without loss of generality, assume $i=1$. Let $r$ and $r'$ be nonforking extensions to $S_1(B)$ of $q$ and $q'$, respectively. So $r$ and $r'$ are distinct nonforking extensions of $p|_{x_1}$ to $S_1(B)$. 

\noindent\textit{Claim}: If $a$ realizes a nonforking extension of $p|_{x_1}$ to $S_1(B)$ then there some $\abar$, realizing a nonforking extension of $p$ to $S_n(B)$, such that $a_1=a$.

\noindent\textit{Proof}: We have $a\ind_C^d B$. By extension, there is $(a_2,\ldots,a_n)\models p(a,x_2,\ldots,x_n)$ such that $a_2,\ldots,a_n\ind_{Ca}^d Ba$. By a transitivity property of dividing (see \cite[Proposition 7.1.6]{TeZi}), we have $\abar\ind^d_C B$, where $\abar=(a,a_2\ldots,a_n)$. So $\abar$ is the desired realization of a nonforking extension of $p$ to $S_n(B)$.\claim

By the claim, $r$ and $r'$ yield distinct nonforking extensions of $p$ to $S_n(B)$.

$(ii)\Rightarrow(i)$: Fix $p\in S_n(C)$. Suppose $\abar$ and $\abar'$ realize distinct (nonforking) extensions of $p$ to $S_n(B)$. Note that $\abar\equiv_C\abar'$ so by quantifier elimination there must be some $1\leq i\leq n$ and $b\in B$ such that $d(a_i,b)\neq d(a'_i,b)$. Therefore $a_i$ and $a_i'$ realize distinct (nonforking) extensions of $p|_{x_i}$ to $S_1(Cb)$.
\end{proof}

Note that the proof of $(i)\Rightarrow(ii)$ holds in any theory where $\ind^d=\ind^f$.

\begin{corollary}
Let $C\seq B\subset\U$ and $p\in S_n(C)$. Then $p$ has a unique nonforking extension to $S_n(B)$ if and only if $p$ has a unique extension to $S_n(B)$.
\end{corollary}
\begin{proof}
The forward direction follows by existence of nonforking extensions. Conversely, combining Theorem \ref{unique ext} and Proposition \ref{one type}, we have that if $p$ has a unique nonforking extension to $S_n(B)$ then for all $1\leq i\leq n$ and $b\in B$, $p|_{x_i}$ has a unique extension to $S_1(Cb)$. As in the proof of Proposition \ref{one type}[$(ii)\Rightarrow (i)$], it follows that $p$ has a unique extension to $S_n(B)$.
\end{proof}

We are now ready to characterize stationary types in $\Th(\cU)$.

\begin{corollary}
Fix $C\subset\U$ and $p\in S_n(C)$. Then $p$ is stationary if and only if $p$ is algebraic, i.e., $p=\tp(\abar/C)$ for some $\abar\in\overline{C}$.
\end{corollary}
\begin{proof}
Let $\abar\models p$. By Proposition \ref{one type}, $p$ is stationary if and only if $\tp(a_i/C)$ is stationary for all $1\leq i\leq n$. By Corollary \ref{in closure}, this is equivalent to $\abar\in\overline{C}$.
\end{proof}

\subsection{Analogs of the Urysohn sphere in discrete logic}  The most direct analog of the Urysohn sphere in discrete logic is the \textit{rational Urysohn sphere}, i.e. the unique countable, universal and ultrahomogeneous metric space of diameter $1$ with rational distances. However, the theory of the rational Urysohn sphere, as a discrete structure in a relational language, is not $\aleph_0$-categorical. Moreover, saturated models contain points of nonstandard distance (e.g. positive infinitesimals). 

A more well-behaved discrete theory is that of the Urysohn space with distances in $\{0,1,2,\ldots,n\}$, for some fixed $n>0$. In \cite{CaWa}, Casanovas and Wagner call this structure the free $n^{\text{th}}$ root of the complete graph, and denote its first order theory by $T_n$. In this case, $T_n$ is $\aleph_0$-categorical and, more importantly, nonstandard distances do not arise in saturated models. Therefore our methods in continuous logic can be translated directly to first order logic to obtain the following results. 
\begin{enumerate}
\item For $n\geq 3$, $T_n$ is $\TP_2$, $\SOP_n$ and $\NSOP_{n+1}$.
\item For complete types in $T_n$, forking and dividing are the same and have the same combinatorial characterization as in $\Th(\cU)$. 
\end{enumerate}
In \cite{CaWa}, Casanovas and Wagner show that $T_n$ is non-simple and without the strict order property. Therefore the first result above is a strengthening of their work.

\section{Final Remarks and Questions}\label{sec:Qs}

In \cite{TeZiSIR}, Tent and Ziegler define a \textit{stationary independence relation} to be an abstract ternary relation $\ind$, on finite subsets of a countable structure, which satisfies the axioms of invariance, monotonicity, transitivity, symmetry, full existence, and stationarity. We refer the reader to \cite{TeZiSIR} for the definitions of these properties.\footnote{Tent and Ziegler refer to what we call full existence as ``existence". Following \cite{Adgeo}, we continue to use \emph{existence} for the axiom discussed in Example \ref{existEx}. Is it easy to see that any ternary relation satisfying full existence and invariance also satisfies existence.} See also \cite{Adgeo} for extensions of these axioms to small subsets of a monster model. 

Tent and Ziegler then define a specific stationary independence relation on the countable rational Urysohn space, which generalizes directly to the complete Urysohn sphere, and thus the monster model $\U$. In particular, this stationary independence relation is given by \textit{free amalgamation of metric spaces}:
$$
A\textstyle\ind_C B~\Leftrightarrow~ d(a,b)=d_{\max}(a,b/C)\text{ for all $a\in A$, $b\in B$.}
$$
In \cite{TeZiSIR}, Tent and Ziegler use this relation to show that the automorphism group of the Urysohn space is boundedly simple, and then in \cite{TeZibdd}, use similar methods to show that the automorphism group of the bounded Urysohn sphere is simple. 

Given the appearance of $d_{\max}$ in our characterization of forking for the Urysohn sphere, it is natural to wonder about the relationship between this stationarity independence relation and nonforking. In fact, this leads to a much more general line of questioning, beginning with the next result, which is proved using calculus of ternary relations in the style of \cite{Adgeo}. 

For the rest of this section, we let $T$ be a (continuous or discrete) complete theory, and $\M$ a monster model of $T$.

\begin{theorem}\label{SIRfork}
Suppose $\ind$ is a stationary independence relation on $\M$. Then, for all $A,B,C \subset \M$, $A\ind_C B$ implies $A\ind^f_C B$ and $B\ind^f_C A$.
\end{theorem}
\begin{proof}
Since $\ind$ satisfies symmetry it suffices to show that $\ind$ implies $\ind^f$. We first show $\ind$ satisfies extension (which is essentially the content of \cite[Remark 1.2(3)]{Adgeo}). Fix $A,B,C,D$ such that $A\ind_C B$ and $BC\seq D$. We want to find $A'\equiv_C A$ such that $A'\ind_C D$. By full existence, there is $A''\equiv_C A$ such that $A''\ind_C BC$, and so $A''\ind_C B$ by monotonicity. By stationarity, we have $A''B\equiv_C AB$, and so $A\ind_C BC$ by invariance. Next, by full existence, there is $A'\equiv_{BC} A$ such that $A'\ind_{BC}D$. By invariance, we also have $A'\ind_C BC$, and so $A'\ind_C D$ by transitivity. 

Finally, since $\ind$ satisfies extension, it suffices to show that $\ind$ implies $\ind^d$. For this, suppose $A\ind_C B$. Fix a $C$-indiscernible sequence $(B_i)_{i<\omega}$, with $B_0=B$. We want to find $A'$ such that $A'B_i\equiv_C AB$ for all $i<\omega$. By full existence there is $A'\equiv_C A$ such that $A'\ind_C\bigcup_{i<\omega}B_i$. By monotonicity, we have $A'\ind_C B_i$ for all $i<\omega$. Given $i<\omega$, fix $A_i$ such that $A_iB_i\equiv_C AB$. For any $i<\omega$, we have $A_i\equiv_C A\equiv_C A'$ and, by invariance, $A_i\ind_C B_i$. By stationarity, we have $A'B_i\equiv_C A_iB_i\equiv_C AB$ for all $i<\omega$, as desired.
\end{proof}

In light of our results on the Urysohn sphere, we ask the following question.

\begin{question}\label{Qmock}
Suppose $T$ has a stationary independence relation. Are forking and dividing the same for complete types?
\end{question}
 
Continuing this line of investigation leads to an even more general question. In particular, we previously noted that if $\ind^f$ coincides with $\ind^d$, then $\ind^f$ satisfies existence. The reverse implication is an open question.

\begin{question}\label{Qexist}
Suppose $\ind^f$ satisfies existence in $T$. Are forking and dividing the same for complete types?
\end{question}

Note that if $\ind$ is a stationary independence relation for $T$ then existence for $\ind$ implies existence for $\ind^f$ by Theorem \ref{SIRfork}. Therefore a positive answer to Question \ref{Qexist} would imply a positive answer to Question \ref{Qmock}. 

This overall line of questioning is also motivated by a result of Chernikov and Kaplan \cite{ChKa}, which says that if $T$ is a complete theory in classical logic, for which nonforking satisfies existence, then, assuming $T$ is also $\NTP_2$, we have a positive answer to Question \ref{Qexist}. In fact, Chernikov and Kaplan show the stronger conclusion that forking and dividing coincide for \textit{formulas}. In general, Question \ref{Qexist} is known to have a negative answer if we replace ``complete types" with ``formulas". In particular, the theory of the generic triangle-free graph is a theory (with $\TP_2$) in which forking and dividing are the same for complete types, but not for formulas (see \cite{Co13}). Moreover, free amalgamation of graphs is a stationary independence relation for the theory of the generic triangle-free graph, and so Question \ref{Qmock} also has a negative answer if asked at the level of formulas. Altogether, since we have shown that the Urysohn sphere is $\TP_2$, this motivates our final question.

\begin{question}\label{Qform}
Are forking and dividing the same for formulas in $\Th(\cU)$? 
\end{question}

\appendix

\section{Continuous Model Theory}\label{appMT}

\renewcommand*{\thesection}{\Alph{section}}

Let $T$ be a complete theory in continuous logic, and $\M$ a monster model of $T$. Throughout this section, we adjust our conventions slightly, and let singletons $a,b,c,x,y,z,\ldots$ denote (possibly infinite) tuples of parameters and variables.

\subsection{The Ehrenfeucht-Mostowski type}\label{EM}
In discrete model theory, the Ehrenfeucht-Mostowski type is a powerful tool for constructing indiscernible sequences. Following \cite[Section 5]{TeZi}, we define this notion for continuous logic.

Given a linear order $I$ and some $n<\omega$, let 
$$
[I]^n=\{(i_1,\ldots,i_n)\in I^n:i_1<\ldots<i_n\}.
$$
Given $\ibar=(i_1,\ldots,i_n)\in[I]^n$ and parameters $a_{i_1},\ldots,a_{i_n}$, let $\abar_{\ibar}=(a_{i_1},\ldots,a_{i_n})$.

\begin{definition}
Let $I$ be a linear order and fix a sequence $\cI=(a_i)_{i\in I}$ in $\M$. The \textbf{Ehrenfeucht-Mostowski type}, or \textbf{EM-type}, of $\cI$ is 
$$
\EM(\cI/A)=\{\vphi(x_1,\ldots,x_n)=0:n<\omega,~\vphi(\xbar)\in\cL_A,~\vphi(a_{\ibar})=0\text{ for all }\ibar\in[I]^n\}.
$$
\end{definition}

We now prove that indiscernible realizations of EM-types can always be found. In discrete logic, this fact is a classic application of Ramsey's Theorem. We adapt the classical proof (e.g \cite[Lemma 5.1.3]{TeZi}) to show that this is true for continuous logic as well. 

\begin{theorem}
Let $I$ and $J$ be infinite linear orders. Suppose $(a_i)_{i\in I}$ is a sequence in $\M$ and $A\subset\M$. Then there is an $A$-indiscernible sequence $(b_j)_{j\in J}$ in $\M$ such that $\EM((a_i)_{i\in I}/A)\seq\EM((b_j)_{j\in J}/A)$. 
\end{theorem}
\begin{proof}
Define the following types, each in the variables $(y_j)_{j\in J}$ and containing parameters from $A$.
\begin{align*}
\Gamma_1 &:= \{|\vphi(\ybar_{\jbar})-\vphi(\ybar_{\kbar})|\leq\epsilon:\vphi(x_1,\ldots,x_n)\in\cL_A,~\jbar,\kbar\in [J]^n,~\epsilon>0\}.\\
\Gamma_2 &:= \{\vphi(\ybar_{\jbar})=0:``\vphi(x_1,\ldots,x_n)=0"\in\EM((a_i)_{i\in I}/A),~\jbar\in [J]^n\}.
\end{align*}
Let $\Gamma=\Gamma_1\cup\Gamma_2$, and suppose that the sequence $(b_j)_{j\in J}$ realizes $\Gamma$. Then $\Gamma_1$ ensures $(b_j)_{j\in J}$ is $A$-indiscernible, and $\Gamma_2$ ensures $\EM((a_i)_{i\in I}/A)\seq\EM((b_j)_{j\in J}/A)$. Therefore, it suffices to show that $\Gamma$ is satisfiable.

Fix a finite subset $\Gamma_0\seq\Gamma$. Suppose
$$
\Gamma_0\cap\Gamma_1=\{|\vphi_t(\ybar_{\jbar_t})-\vphi_t(\ybar_{\kbar_t})|\leq\epsilon_t:1\leq t\leq m\},
$$
for some fixed $m>0$. Let $n_t$ be the length of $\jbar_t$, and set $N=\max\{n_1\ldots,n_t\}$. Fix $r>0$ such that $\frac{1}{r}\leq\epsilon_t$ for all $1\leq t\leq m$, and set $\Sigma=\{[\frac{l}{r},\frac{l+1}{r}]:0\leq l<r\}$. Define $F:[I]^N\longrightarrow\Sigma^m$ such that, given $\ibar\in[I]^N$ and $1\leq t\leq m$,
$$
\vphi_t(\abar_{\ibar})\in F(\ibar)(t).
$$
By Ramsey's Theorem, there is an infinite subset $I_*\seq I$ , and some $\sigma\in\Sigma^m$, such that $F(\ibar)=\sigma$ for all $\ibar\in[I_*]^N$.

Let $J_0\seq J$ be finite such that if $y_j$ occurs in some condition of $\Gamma_0$ then $j\in J_0$. Pick $\{i_j:j\in J_0\}\seq I_*$ such that, given $j,k\in J_0$, $j<k$ implies $i_j<i_k$. Given $j\in J_0$, let $b_j=a_{i_j}$. We show that $(b_j)_{j\in J_0}$ realizes $\Gamma_0$.

First, if $\vphi(x_1,\ldots,x_n)$ is an $\cL_A$-formula and $\jbar\in[J_0]^n$, with $``\vphi(y_{\jbar})=0"\in\Gamma_0\cap\Gamma_2$, then $\ibar_{\jbar}\in[I_*]^n$ and so $\vphi(\bbar_{\jbar})=0$ since $``\vphi(x_1,\ldots,x_n)=0"\in\EM((a_i)_{i\in I}/A)$. Finally, fix $1\leq t\leq m$ and let $n=n_t$. Given $\jbar,\kbar\in[J_0]^n$, we have $\ibar_{\jbar},\ibar_{\kbar}\in[I_*]^n$. Since $I_*$ is infinite, we may find $\ibar_{\jbar'},\ibar_{\kbar'}\in [I_*]^N$ such that $\ibar_{\jbar}\seq\ibar_{\jbar'}$ and $\ibar_{\kbar}\seq\ibar_{\kbar'}$. Then
$$
F(\ibar_{\jbar'})(t)=\sigma(t)=F(\ibar_{\kbar'}),
$$
and so, for some $0\leq l<r$, we have $\vphi_t(\bbar_{\jbar}),\vphi_t(\bbar_{\kbar})\in[\frac{l}{r},\frac{l+1}{r}]$. Therefore, it follows that $|\vphi_t(\bbar_{\jbar})-\vphi_t(\bbar_{\kbar})|\leq\frac{1}{r}\leq\epsilon_t$, as desired.
\end{proof}

\subsection{Forking and Dividing}

We work from the definitions of forking and dividing given in Section \ref{sec:MT}. The goal of this section is to prove Theorem \ref{FANDD} and complete the proof of Theorem \ref{ext reduct}. Given a finite partial type $\pi(x)$, we use the notation $\max\pi(x)$ to denote the formula $\max\{\vphi(x):``\vphi(x)=0"\in\pi(x)\}$. 

\begin{lemma}\label{divprop1}
Fix $C\subset\M$ and suppose $\pi(x,b)$ is a type. If $\pi(x,b)$ divides over $C$, then there are formulas $\vphi_1(x,b),\ldots,\vphi_m(x,b)$ such that $``\vphi_t(x,b)=0"\in\pi(x,b)$, for all $1\leq t\leq m$, and $\max\{\vphi_1(x,b),\ldots,\vphi_m(x,b)\}$ divides over $C$.
\end{lemma}
\begin{proof}
Suppose $\pi(x,b)\models``\vphi(x,b')=0"$ and $\vphi(x,b')$ divides over $C$, witnessed by a sequence $(b'_i)_{i<\omega}$, an integer $k>0$, and some $\epsilon>0$. By compactness, there is a finite subset $\pi_0(x,b)\seq\pi(x,b)$ such that $\pi_0(x,b)\models``\vphi(x,b')\leq\frac{\epsilon}{2}"$. Let $\theta(x,b)=\max\pi_0(x,b)$. By compactness again, we may fix $\delta>0$ such that $``\theta(x,b)\leq\delta"\models``\vphi(x,b')\leq\epsilon"$. Given $i<\omega$, let $b_i$ be such that $bb'\equiv_C b_ib'_i$. Then $``\theta(x,b_i)\leq\delta"\models``\vphi(x,b'_i)\leq\epsilon"$. It follows that $\{\theta(x,b_i)\leq\delta:i<\omega\}$ is $k$-unsatisfiable. Altogether $\theta(x,b_i)$ divides over $C$.
\end{proof}

We can now prove Theorem \ref{FANDD}. 

\begin{proof}[Proof of Theorem \ref{FANDD}]
Part $(a)$: Fix a subset $C\subset\M$ and a partial type $\pi(x,b)$. Given a formula $\vphi(x,y)$, an integer $k>0$, and some $\epsilon>0$, let 
$$
\vphi^\epsilon_k(y_1,\ldots,y_k):=\sup_x\min_{1\leq j\leq k}(\epsilon\dotminus\vphi(x,y_j)).
$$

Suppose $\pi(x,b)$ divides over $C\subset\M$. Then, by Lemma \ref{divprop1}$(a)$, there is a finite subset $\pi_0(x,b)$ such that, if $\vphi(x,b)=\max\pi_0(x,b)$, then $\vphi(x,b)$ divides over $C$, say witnessed by a sequence $(b_i)_{i<\omega}$, an integer $k>0$, and some $\epsilon>0$. Then $\{\vphi(x,b_i)\leq\epsilon:i<\omega\}$ is $k$-unsatisfiable, and so $``\vphi^\epsilon_k(\ybar)=0"\in\EM((b_i)_{i<\omega}/C)$. Note also, for any $\cL_C$-formula $\psi(y)$, if $\psi(b)=0$ then $``\psi(y_1)=0"\in\EM((b_i)_{i<\omega}/C)$. By Theorem \ref{EM}, we may replace $(b_i)_{i<\omega}$ with a $C$-indiscernible sequence realizing $\EM((b_i)_{i<\omega}/C)$. By our observations, we still have $b_i\equiv_C b$ for all $i<\omega$, and so, after conjugating by an automorphism of $\M$, we may assume $b_0=b$. Moreover, since $``\vphi^\epsilon_k(\ybar)=0"\in\EM((b_i)_{i<\omega}/C)$, it follows that  $\{\vphi(x,b_i)=0:i<\omega\}$ is unsatisfiable, and so $\bigcup_{i<\omega}\pi(x,b_i)$ is unsatisfiable.

Conversely, suppose there is a $C$-indiscernible sequence $(b_i)_{i<\omega}$, with $b_0=b$, such that $\bigcup_{i<\omega}\pi(x,b_i)$ is unsatisfiable. By compactness, there is a finite subset $\pi_0(x,b)\seq\pi(x,b)$, indices $i_1<\ldots<i_k$, and some $\epsilon>0$ such that, if $\vphi(x,b)=\max\pi_0(x,b)$ then $\{\vphi(x,b_{i_1})\leq\epsilon,\ldots,\vphi(x,b_{i_k})\leq\epsilon\}$ is unsatisfiable. Then $\vphi^\epsilon_k(b_{i_1},\ldots,b_{i_k})=0$, and so $\vphi^\epsilon_k(b_{j_1},\ldots,b_{j_k})=0$ for all $j_1<\ldots<j_k<\omega$ by indiscernibility. It follows that $\{\vphi(x,b_i)\leq\frac{\epsilon}{2}:i<\omega\}$ is $k$-unsatisfiable. Altogether, $\vphi(x,b)$ divides over $C$ and $\pi(x,b)\models``\vphi(x,b)=0"$, as desired. 

Part $(b)$: Fix a subset $C\subset\M$ and a partial type $\pi(x,b)$. The forward direction follows exactly as in discrete logic, so we prove the converse. Suppose there is some $D\supseteq Cb$ such that any extension of $\pi(x,b)$ to a complete type over $D$ divides over $C$. Let $\Sigma$ be the collection of $\cL_D$-formulas $\theta(x)$, which divide over $C$. For each $\theta\in\Sigma$, we use Lemma \ref{divprop2} to fix $\epsilon_\theta>0$ such that $\theta(x)\dotminus\epsilon_\theta$ divides over $C$. Define $p_0(x)=\pi(x,b)\cup\{\theta(x)\geq\epsilon_\theta:\theta\in\Sigma\}$, and suppose, towards a contradiction, that $p_0(x)$ is satisfiable. Then we can extend $p(x)$ to a complete type $p(x)$ over $D$, which divides over $C$ by assumption. By Lemma \ref{divprop1}, there is some $\cL_D$-formula $\theta(x)$ such that $``\theta(x)=0"\in p(x)$ and $\theta(x)$ divides over $C$. But then $``\theta(x)\geq\epsilon_\theta"\in p_0(x)$, which is a contradiction. Therefore $p_0(x)$ is inconsistent. By compactness, there are $\theta_1,\ldots,\theta_m\in\Sigma$ such that, if $\psi_t(x):=\theta_t(x)\dotminus\epsilon_{\theta_t}$, then
\begin{equation*}
\pi(x,b)\models``\min\{\psi_1(x),\ldots,\psi_m(x)\}=0"\qedhere
\end{equation*}
\end{proof}

Our final goal is to complete the proof of Theorem \ref{ext reduct}, which can now be done exactly as in the discrete case. Recall that a ternary relation $\ind$ satisfies \textit{extension} if for all $A,B,C,D\subset\M$, if $A\ind_C B$ and $BC\seq D$, then there is some $A'\equiv_{BC}A$ such that $A'\ind_C D$.

\begin{exercise}\label{forklem}$~$
\begin{enumerate}[$(a)$]
\item Using the characterization of forking given by Theorem \ref{FANDD}$(b)$, follow the proof of \cite[Lemma 3.1]{Adgeo} to show $\ind^f$ satisfies extension in any continuous theory.
\item Alternatively, use the original ``syntactic" definition of forking to prove the following continuous analog of a classical exercise: if $C\subset\M$ and $M\supseteq C$ is $\chi(C)^+$-saturated, then forking and dividing over $C$ coincide for complete types over $M$. Then use Theorem \ref{FANDD}$(b)$ to conclude that $\ind^f$ satisfies extension.
\item Use the fact that $\ind^f$ satisfies extension to complete the proof of Theorem \ref{ext reduct}. In particular, prove that $\ind^d$ coincides with $\ind^f$ if and only if $\ind^d$ satisfies extension.
\end{enumerate}
\end{exercise}

\bibliography{D:/Gabe/UIC/gconant/Math/BibTeX/biblio}
\bibliographystyle{amsplain}

\end{document}